\documentclass[11pt]{amsart}
\usepackage{amssymb}
\usepackage{amsfonts}
\usepackage{mathrsfs}
\usepackage{latexsym}
\usepackage{graphicx}
\usepackage{amscd,amssymb,amsmath,amsbsy,amsthm,amsfonts}
\usepackage[all]{xy}
\usepackage[colorlinks,plainpages,backref,urlcolor=blue]{hyperref}
\usepackage{verbatim}
\usepackage{enumerate}

\usepackage [english]{babel}
\usepackage [autostyle, english = american]{csquotes}
\MakeOuterQuote{"}

\newtheorem{theorem}{Theorem}[section]
\newtheorem{lemma}[theorem]{Lemma}

\newtheorem{corollary}[theorem]{Corollary}
\newtheorem{prop}[theorem]{Proposition}

\theoremstyle{definition}
\newtheorem{definition}[theorem]{Definition}

\newtheorem{remark}[theorem]{Remark}

\newtheorem{remarks}[theorem]{Remarks}

\newtheorem{thm}{Theorem}

\numberwithin{equation}{section}

\newcommand{\Z}{\mathbb{Z}}
\newcommand{\Q}{\mathbb{Q}}

\newcommand{\G}{\mathbb{G}}
\newcommand{\PP}{\mathbb{P}}
\newcommand{\sO}{\mathscr{O}}

\newcommand{\QQ}{\mathbb{Q}}
\newcommand{\ZZ}{\mathbb{Z}}

\newcommand{\calX}{\mathcal{X}}
\newcommand{\sA}{\mathcal{A}}
\newcommand{\sB}{\mathcal{B}}
\newcommand{\sX}{\mathcal{X}}
\newcommand{\sY}{\mathcal{Y}}

\newcommand{\sE}{\mathcal{E}}
\newcommand{\sF}{\mathcal{F}}
\newcommand{\sI}{\mathcal{I}}
\newcommand{\clf}{\text{clfl}}
\newcommand{\fl}{\text{fl}}

\newcommand{\bS}{\mathbb{S}}

\newcommand{\0}{\emptyset}

\newcommand{\op}{\text{op}}
\newcommand{\Ab}{\mathbf{Ab}}

\newenvironment{romenum}
{

\begin{enumerate}}{\end{enumerate}}

\DeclareMathOperator{\SheafHom}{\mathbf{Hom}}
\newcommand{\F}{\mathbf{F}}

\DeclareMathOperator{\Hom}{Hom}
\DeclareMathOperator{\End}{End}

\DeclareMathOperator{\Ell}{Ell}
\DeclareMathOperator{\sEll}{\mathcal{E}\ell\ell}

\DeclareMathOperator{\TT}{T}
\DeclareMathOperator{\Laz}{Laz}
\DeclareMathOperator{\Comm}{Comm}
\DeclareMathOperator{\Td}{Td}
\DeclareMathOperator{\Sm}{Sm}
\DeclareMathOperator{\CH}{CH}
\DeclareMathOperator{\Spec}{Spec}

\DeclareMathOperator{\SH}{SH}
\DeclareMathOperator{\MGL}{MGL}
\DeclareMathOperator{\MU}{MU}

\DeclareMathOperator{\codim}{codim}
\DeclareMathOperator{\OCT}{\mathbf{OCT}}
\DeclareMathOperator{\MCT}{\mathbf{MOCT}}
\DeclareMathOperator{\oct}{oct}

\newcommand{\<}{\langle}
\renewcommand{\>}{\rangle}

\makeatletter
\let\@wraptoccontribs\wraptoccontribs
\makeatother

\title[Algebraic elliptic cohomology theory and flops \uppercase\expandafter{\romannumeral1}]%
{Algebraic elliptic cohomology theory and flops \uppercase\expandafter{\romannumeral1}}
\date{\today}
\date{\today}

\author[M.~Levine]{Marc~Levine}
\address{Universit\"at Duisburg-Essen,
Fakult\"at Mathematik, Campus Essen, 45117 Essen, Germany}
\email{marc.levine@uni-due.de}

\author[Y.~Yang]{Yaping~Yang}
\address{Department of Mathematics,
Northeastern University,
Boston, MA 02115, USA}
\email{yang.yap@husky.neu.edu}

\author[G.~Zhao]{Gufang~Zhao}
\address{Department of Mathematics,
Northeastern University, Boston, MA 02115, USA}
\email{zhao.g@husky.neu.edu}

\contrib[\vskip1ex  with an appendix by]{Jo\"el Riou}
\address{D\'epartement  de Math\'ematiques, Universit\'e Paris-Sud,
B\^at.~425, 91405, Orsay, France}
\email{joel.riou@math.u-psud.fr}

\begin{document}

\begin{abstract}
We define the algebraic elliptic cohomology theory coming from Krichever's elliptic genus as an oriented cohomology theory on smooth varieties over an arbitrary perfect field. We show that in the algebraic cobordism ring with rational coefficients, the ideal generated by differences of classical flops coincides with the kernel of Krichever's elliptic genus. This generalizes a theorem of B.~Totaro  in the complex analytic setting.
\end{abstract}

\thanks{The last two authors are grateful to Universit\"at Duisburg-Essen for hospitality and excellent working conditions. The first author thanks the Humboldt Foundation for its support.}

\maketitle

\tableofcontents

\section*{Introduction}
The notion of complex oriented cohomology theory in the category of smooth manifolds was introduced by Quillen  \cite{Q} in his study of the universal  complex oriented cohomology theory, complex cobordism. The existence of a formal group law associated to a complex oriented theory allowed Quillen to define an isomorphism of the complex cobordism ring with the Lazard ring $\Laz$, the underlying ring of the universal formal group law $F_{\Laz}(u,v)\in \Laz[\![u,v]\!]$.

Levine and Morel \cite[Definition 1.1.2]{LM} have defined the notion of an oriented cohomology theory on the category of smooth quasi-projective schemes $\Sm_k$ over a perfect field $k$ as an algebraic analog of complex oriented cohomology, without however requiring an excision property or a Mayer-Vietoris sequence. When the underlying field $k$ has characteristic zero, the universal oriented cohomology theory, $X\mapsto \Omega^*(X)$,  exists and is called {\em algebraic cobordism}.  Just as for complex cobordism,  the coefficient ring $\Omega^*(k)$ of  algebraic cobordism is isomorphic to the Lazard ring.  The existence of a universal theory over an arbitrary perfect field is not at present known. We partially fill this gap using motivic homotopy theory to contruct the universal theory with $\Q$-coefficients, relying on the algebraic cobordism spectrum $\MGL$ and the  theorem of Hopkins-Morel, recently established in detail by Hoyois \cite{Hoyois}.

We apply this machinery to give algebraic versions of {\em elliptic cohomology}. Given a family of elliptic curves over some ring $R$, there is a formal group law over the ring $R$, coming from the additive structure of the elliptic curves and a choice of parameter along the zero section over $\Spec R$. The corresponding cohomology theory $X\mapsto R^*(X)$, an algebraic elliptic cohomology theory,  is the main subject of this paper.

Such group laws are often constructed using a so-called {\em elliptic genus} rather than an explicit construction of the family of elliptic curves. When $k = \mathbb{C}$,  a particular elliptic genus has a crucial property, called the rigidity property, proved by Krichever in \cite{K} and H\"{o}hn in \cite{H}.

The elliptic genus also arises in the study of Chern numbers of singular varieties. Inspired by H\"{o}hn's work, Totaro showed in \cite[Theorem 4.1]{T} that the kernel of the elliptic genus coincides with the ideal in the complex cobordism ring $\MU^*(\mathbb{C})$ generated by the differences of classical flops. As a corollary, the characteristic numbers which can be defined for singular varieties in a fashion which is compatible with small resolutions, are exactly the specializations of the elliptic genus. 
Wang proved in \cite[Theorem~D]{W} that the ideal in the complex cobordism ring $\MU^*(\mathbb C)$ generated by differences of
flops equals to the ideal generated by the differences of $K$-equivalent varieties.

It is worth mentioning that in Totaro's work, the proofs take place in the setting of weakly-complex manifolds and use topological constructions which  do not lend themselves to the situation over a field of positive characteristic.

In this paper, we study the algebraic version of the elliptic cohomology  over an arbitrary perfect field, and consider as well as the question of the existence of a corresponding motivic oriented cohomology theory representing elliptic cohomology.

The underlying ring of the elliptic formal group law corresponding to  Krichever's elliptic genus is a certain subring of $\ZZ[a_1,a_2,a_3,\frac{1}{2}a_4]$; there is an explicit descriptions of the four generators $a_i$ (see \cite{BB}) as elements in a formal power series ring $ \QQ(\!(e^{2\pi iz})\!)[\![e^{2i\pi\tau},k]\!]$. In order to make this formal group law Landweber exact, we need to enlarge the coefficient ring. For an integer $n\ge1$, let $\Ell[1/n]=\ZZ[1/n][a_1,a_2,a_3,a_4][\Delta^{-1}]$, where $\Delta$ is the discriminant (see \eqref{discr}). 

For a field $k$, the {\em exponential characteristic} of $k$ is defined to be 1 if $k$ has characteristic zero, and is $\text{char} k$ if  $\text{char} k>0$.

\begin{thm}[Theorem~\ref{thm: ellip_exact}]\label{thm: ellip_exact_intro} Let $k$ be a perfect field of exponential characteristic $p$.
The oriented cohomology theory on $\Sm_k$ sending $X\mapsto \MGL^*(X)\otimes_{\Laz}\Ell[1/2p]$ is  represented by a motivic oriented cohomology theory  on $\Sm_k$.
\end{thm}

Now we can state our main theorem in this paper as follows. Let $\MGL^*_\QQ(k):=\MGL^*(k)\otimes_\ZZ \QQ$ be the algebraic cobordism ring with $\QQ$ coefficients, and $\Ell^*_\QQ(k)$ be the elliptic cohomology ring with $\QQ$ coefficients.
\begin{thm}[Proposition~\ref{prop: flop in ell}, Proposition~\ref{prop: deg<5}, and Corollary~\ref{cor: injModFlop}]\label{thm:FlopMain}
The kernel of the algebraic elliptic genus $\phi_\QQ: \MGL^*_\QQ(k)\to \Ell^*_\QQ(k)$ is generated by the differences of classical flops, and its image is the polynomial ring $\QQ[a_1,a_2,a_3,a_4]$.
\end{thm}
The basic idea underlying our approach is to use the {\em double-point cobordism} of Levine-Pand\-hari\-pande \cite{LP} to give a simple explicit description of the difference of two flops in the algebraic cobordism ring $\MGL^*(k)$ (see Proposition~\ref{prop:Diff_Omega} for a more precise statement), replacing Totaro's topological  constructions. The vanishing of the difference of two flops in elliptic cohomology is then a consequence of a classical identity satisfied by the $\sigma$ function (see Proposition~\ref{prop: flop in ell}).

The paper is organized as follows:  In \S\ref{sec:Cohomology} we recall some foundational material concerning oriented cohomology and motivic oriented cohomology. We show that $\MGL^*_\QQ$ is the universal oriented cohomology theory on $\Sm_k$ in \S\ref{sec:Univ}.   In \S\ref{sec: Landweber} we give the construction of elliptic cohomology as an oriented cohomology theory on $\Sm_k$ for $k$ an arbitrary perfect field, and we prove our main result on the existence of a motivic oriented theory representing elliptic cohomology. In \S\ref{sec:CobordFlop}  we introduce the double-point cobordism, use this theory to give an explicit description of the difference of two flops in $\MGL^*$-theory and show that the difference of two flops vanishes in elliptic cohomology. In \S\ref{sec:FlopIdeal}  we use H\"ohn's  and Totaro's algebraic computations to prove our main result (Theorem~\ref{thm:FlopMain}) on rational elliptic cohomology and its relation to $\MGL^*_\Q$. We conclude in \S\ref{sec:Birat} with an application of algebraic cobordism to the Chern numbers of  smooth projective symplectic varieties.

We include an appendix by Jo\"el Riou in which it is shown that the infinite suspension spectrum of a smooth $k$-scheme is a strongly dualizable object in $\SH(k)$ after inverting the exponential characteristic. This was known for a field of characteristic zero, or for an arbitrary field after passing to $\Q$-coefficients, but the result in this more precise form is apparantly not in the literature.

\section{Oriented cohomology}\label{sec:Cohomology}

\subsection{Oriented cohomology theories and algebraic cobordism}\label{subsec: prelim}
In this subsection we collect preliminary notions and results we will use. The main goal is to fix the notations and conventions.

Recall that a \textit{formal group law} over a commutative ring $R$ with unit is an element $F(u,v)\in R[\![u,v]\!]$ satisfying the following conditions
\begin{enumerate}
  \item $F(u,0)=u$, $F(0,v)=v$;
    \item $F(u,v)=F(v,u)$;
  \item $F(F(u,v),w)=F(u,F(v,w))$.
\end{enumerate}
Lazard \cite{L}  pointed out the existence of a universal formal group law $(\Laz, F_{\Laz})$.
  We give $\Laz$ the grading so that, with $u$ and $v$ of degree $+1$, the power series $F(u,v)$ is homogeneous of degree 1, and we will always consider formal group laws with coefficients in a graded ring $R$, so that the associated classifying homomorphism $\phi_{(F,R)}:\Laz\to R$  preserves the grading.

Let $k$ be a field, and $\Sm_k$ be the category of smooth, quasi-projective schemes over $k$.
Let $\Comm$ denote the category of commutative, graded rings with unit.
\begin{definition}[ \hbox{\cite[Definition 1.1.2]{LM}}]\label{def: cohom_th}
An oriented cohomology theory on $\Sm_k$ is given by:
\begin{description}
  \item[D1]  A functor $A^*:\Sm_k^\op\to \Comm$, $X\mapsto A^*(X)$, $(f:Y\to X)\mapsto f^*:A^*(X)\to A^*(Y)$.
  \item[D2]   For any projective morphism $f: Y\to X$ in $\Sm_k$ of relative codimension $d$, a homomorphism of graded $A^*(X)$-modules $f_*: A^*(Y) \to A^{*+d}(X)$.
\end{description}
These data satisfy the axioms described in  \cite[\hbox{\it loc. cit.}]{LM}.

A morphism $f:A^*\to B^*$ of oriented cohomology theories is a natural transformation of the functors in (D1), which is also natural with respect to the maps $f_*$ in (D2). 
\end{definition}
For $p:X\to \Spec k$ a smooth projective $k$-scheme of dimension $d$, we have the unit $1_X\in A^0(X)$ and its push-forward $[X]:=p_{X*}(1_X)\in A^{-d}(k)$.

For $L\to X$ a line bundle over some $X\in \Sm_k$, the first Chern class of $L$ in the theory $A$, $c_1^A(L)$, is defined by $c_1^A(L):=s^*s_*(1_X)$, where $s:X\to L$ is the zero section.  For an oriented cohomology theory $A^*$, there is a unique power series $F_A(u,v)\in A^*(k)[\![u,v]\!]$ satisfying
\[
F_A(c_1(L),c_1(M))=c_1(L\otimes M)
\]
for each pair line bundles $L$ and $M$ on a scheme $X\in \Sm_k$; moreover,   $F$ defines a formal group law over the ring $A^*(k)$  \cite[Lemma 1.1.3]{LM}. We denote the ring $A^*(X)\otimes\QQ$ simply by $A^*_\QQ(X)$.

In case $k$ has characteristic zero, the universal oriented cohomology theory $\Omega^*$ has been constructed by Levine and Morel.
\begin{theorem}[\hbox{\cite[Theorem 4.3.7, Theorem 7.1.3]{LM}}]\label{thm:OmegaUniv}
Assume the base field $k$ has characteristic zero.
\begin{enumerate}[i)]
  \item There is a universal oriented cohomology theory $\Omega^*$  on $\Sm_k$.
  \item The canonical ring homomorphism $\Laz \to \Omega^*(k)$ induced by the formal group law $F_\Omega$ of the algebraic cobordism $\Omega^*$ is an isomorphism.
\end{enumerate}
\end{theorem}
For $k$ of characteristic zero and $A$ an oriented cohomology theory on $\Sm_k$, we let
\begin{equation}\label{eqn:classify}
\Theta_A:\Omega^*\to A^*
\end{equation}
be the classifying map given by Theorem~\ref{thm:OmegaUniv}(i).

When the base field $k$ has positive characteristic,  there is at present no construction of a universal oriented theory known. However, the construction of $\Omega^*(X)$ in \cite{LM} leads to a notion of a ``universal oriented Borel-Moore $\Laz$-functor on $\Sm_k$ of geometric type'' (see \cite[Definition 2.2.1]{LM}), and $\Omega^*$ is the universal such theory \cite[Theorem 2.4.13]{LM}.

We recall Quillen's formula:
\begin{theorem}[Quillen]
Let $X$ be a smooth quasi-projective variety, over a field $k$,  $V$ an $n$-dimensional vector bundle on $X$, and $\pi: \PP_X(V) \to X$ be the corresponding projective bundle. Let $A$ be an oriented cohomology theory on $\Sm_k$ with formal group law $F_A$ and take $f(t) \in A^*(X)[\![t]\!]$. Then
\begin{equation}\label{Quill-formula}
\pi_*(f(c_1(\sO(1))))=\sum_i\frac{f(-_{A}\lambda_i)}{\prod_{j\neq i}(\lambda_j-_{A}\lambda_i)},
\end{equation}
where $\lambda_i$ are the Chern roots of $V$,  $x+_Ay:=F_A(x, y)$ and $-_Ax$ is the inverse for $F_A$.
\end{theorem}
A proof of this theorem in the context of complex cobordism can be found in \cite[page 50]{V}. The proof goes through word for word in our setting, so we will not repeat it here.

\subsection{Motivic oriented cohomology theories}

In positive characteristic, the universal oriented cohomology theory is not available. In \S\ref{sec:Univ} we will use  motivic homotopy theory to construct the universal theory with $\Q$-coefficients; we recall some of the basic notions here; for details on the constructions and basic properties we use, we refer the reader to  \cite{RO, MV, PPR1, VoevICM}.

Let $\SH(k)$ be the motivic stable homotopy category of $\PP^1$-spectra. One has the {\em infinite $\PP^1$-suspension functor} $\Sigma^\infty_{\PP^1}(-)_+:\Sm_k\to \SH(k)$ and suspension functors 
\[
\Sigma^n_{S^1}, \Sigma_{\G_m}^n, \Sigma_{\PP^1}^n:\SH(k)\to \SH(k)
\]
for all $n\in \ZZ$. $\SH(k)$ is a triangulated tensor category with translation $\Sigma_{S^1}$,  tensor product $(\sE, \sF)\mapsto \sE\wedge \sF$ and unit the motivic sphere spectrum $\bS_k:=\Sigma^\infty_{\PP^1}\Spec k_+$.  An object $\sE$ of $\SH(k)$ and integers $n,m$ define a functor $\sE^{n,m}:\Sm_k^\op\to \Ab$ by
\[
\sE^{n,m}(X):=\Hom_{\SH(k)}(\Sigma^\infty_{\PP^1} X_+, \Sigma_{S^1}^{n-m}\Sigma_{\G_m}^m \sE).
\]

We use the notions of a {\em commutative $\PP^1$ ring spectrum}, $(\sE, \mu:\sE\wedge\sE\to \sE, \epsilon:\bS_k\to \sE)$ and that of a (Chern) orientation $\vartheta\in \sE^{2,1}(\PP^\infty/0)$ from  \cite[\S1]{PPR2}.
 
These notions may be extended slightly to the setting of   {\em weak} commutative $\PP^1$ ring spectra and  oriented weak commutative ring spectra. Here, a weak commutative $\PP^1$ ring spectrum is a triple $(\sE, \mu:\sE\wedge\sE\to \sE, \epsilon:\bS_k\to \sE)$ as above that defines a commutative monoid object in $\SH(k)/\text{ph}$, where $\text{ph}$ is the two-sided ideal of {\em phantom maps} in $\SH(k)$, that is, maps $f:\sA\to \sB$ in $\SH(k)$ such that $f\circ i=0$ for all maps $i:K\to \sA$, where $K$ is a compact object in $\SH(k)$; see \cite[\S8, 1st para.]{NSO}. An orientation is the same as for commutative ring spectra, in other words, $\vartheta\in \sE^{2,1}(\PP^\infty/0)$ is an element whose restriction to 
$\sE^{2,1}(\PP^1/0)$ is the image of the unit $\epsilon\in \sE^{0,0}(k)$ via the suspension isomorphism $\sE^{0,0}(k)\cong
\sE^{2,1}(\PP^1/0)$. This leads to the following definition.

 \begin{definition} A {\em motivic oriented cohomology theory  on $\Sm_k$} is a weak commutative $\PP^1$ ring spectrum $(\sE,\mu, \epsilon)$ in $\SH(k)$  together with an orientation $\vartheta\in \sE^{2,1}(\PP^\infty/0)$.  
A morphism of motivic oriented cohomology theories 
\[
\phi:(\sE,\mu, \epsilon,\vartheta)\to(\sE',\mu', \epsilon',\vartheta') 
\]
 is a morphism  $\phi:\sE\to \sE'$ in $\SH(k)$ that defines a map of monoid objects in $\SH(k)/\text{ph}$ and satisfies 
 $\phi_*(\vartheta)=\vartheta'$ in $\sE^{\prime 2,1}(\PP^\infty/0)$.
\end{definition}

Relying on \cite[Theorem 3.5]{PaninSmirnov}, it is described in \cite[\S2]{PPR2} how an orientation $\vartheta$ for a  commutative $\PP^1$ ring spectrum $\sE\in \SH(k)$  gives rise to functorial pushforward maps $f_*:\sE^{a,b}(Y)\to \sE^{a-2d, b-d}(X)$ for each projective morphism $f:Y\to X$ in $\Sm_k$, where $d$ is the relative dimension of $f$. In addition, the data of the contravariant functor $X\mapsto \sE^*(X):=\sE^{2*,*}(X)$ from $\Sm_k$ to $\Comm$, together with the maps $f_*$,   define an oriented cohomology theory on $\Sm_k$. Moreover, this construction is functorial in the motivic oriented theory: if  $\phi:(\sE,\mu_\sE,1_\sE,\vartheta_\sE)\to (\sF,\mu_\sF,1_\sF,\vartheta_\sF)$ is a morphism of motivic oriented cohomology theories, then the induced natural transformation $X\mapsto \phi_X:\sE^*(X)\to \sF^*(X)$, $X\in \Sm_k$, defines a morphism of oriented cohomology theories.

 The arguments used in \cite[\hbox{\it loc. cit.}]{PaninSmirnov} use only the structure of $\sE$ as a cohomology theory on pairs $(X, U)$, $U\subset X$ open, $X\in \Sm_k$, so the constructions and arguments work just as well for oriented weak commutative $\PP^1$ ring spectra. 
 
Thus, letting $\OCT(k)$ denote the category of oriented cohomology theories on $\Sm_k$ and $\MCT(k)$ the category of motivic oriented cohomology theories on $\Sm_k$, we have the functor
 \[
 \oct:\MCT(k)\to\OCT(k).
 \]
 sending $(\sE, 1, \mu,\vartheta)$ to $\sE^*$ and $\phi:(\sE,\mu_\sE,1_\sE,\vartheta_\sE)\to (\sF,\mu_\sF,1_\sF,\vartheta_\sF)$ to $\phi:\sE^*\to\sF^*$.
 
 We will simplify the notation by dropping auxiliary data, writing for instance $(\sE,\vartheta)$ for a motivic oriented cohomology theory, $A^*$ for an oriented cohomology theory and $\sE^*$ for the oriented cohomology theory $\oct((\sE,\vartheta))$; we may even drop the orientation $\vartheta$ from the notation if the context makes the meaning clear. For a motivic oriented cohomology theory $(\sE,\vartheta)$, we let $\phi_\sE:\Laz\to \sE^*(k)$ be the homomorphism classifying the formal group law of the oriented cohomology theory $\sE^*$. We say that a  motivic oriented cohomology theory $(\sE,\vartheta)$ {\em represents} an oriented cohomology theory $A^*$ if there is an isomorphism of oriented cohomology theories $\alpha:A^*\to \sE^*$. 

\begin{remarks}\label{rem:MGLOrientedCoh}\ \\
1. The notion of a  motivic oriented cohomology theory on $\Sm_k$ is referred to as an ``oriented ring cohomology theory'' in  \cite{PPR2}. We find this too similar to the term ``oriented cohomology theory'', hence our relabelling.\\
2.  The algebraic cobordism $\PP^1$-spectrum $\MGL \in \SH(k)$ with its canonical orientation $\vartheta_{\MGL}$ is the universal motivic oriented cohomology theory  on $\Sm_k$. See \cite{PPR1, VoevICM} for the construction of $(\MGL,\vartheta_{\MGL})$ and \cite{PPR2}  for the proof of universality\footnote{In \cite{PPR2} the more restrictive notion of motivic oriented cohomology theory is used, in that the objects are assumed to the commutative ring spectra rather than weak commutative ring spectra.  As pointed out in \cite[\S1, comments after Theorem 1.4]{L15}, the proofs and results of \cite{PPR2}  carry over for
 oriented weak commutative $\PP^1$ ring spectra without change, as all the arguments use only the resulting cohomology theories on pairs of smooth schemes of finite type}.\\
3. Let $k$ be a field of characteristic zero. 
Since $\MGL^*$ is an oriented cohomology theory,  we have the canonical comparison morphism  \eqref{eqn:classify} of oriented cohomology theories on $\Sm_k$,  $\Theta_{\MGL}: \Omega^*\to \MGL^*$, which is an isomorphism by  \cite[Theorem 3.1]{Lev_compare}. Thus, for a field of characteristic zero,  $\MGL^*$ is the universal oriented cohomology theory on $\Sm_k$.
\end{remarks}

\section{Universality}\label{sec:Univ} We have noted that $\MGL^*$ (or $\Omega^*$)  is the universal oriented cohomology theory on $\Sm_k$, if $k$ has characteristic zero. In this section, we prove a number of weaker universality statements for an arbitrary perfect field $k$.

\subsection{Specialization of the formal group law}

From the oriented cohomology theory $\MGL^*$, and a given formal group law $F(u,v)\in R[\![u,v]\!]$ such that the exponential characteristic $p$ of $k$ is invertible in $R$, we may construct an oriented cohomology theory $R^*$ on $\Sm_k$ with $R^*(k)=R$ and formal group law $F$ as follows.

For  $X\in \Sm_k$, pullback by the structure map makes $ \MGL^*(X)$ an $\MGL^*(k)$-algebra; the pull-back maps $f^*:\MGL^*(X)\to \MGL^*(Y)$ are  $\MGL^*(k)$-algebra homomorphisms and the projective push forward maps $f_*$ are $\MGL^*(k)$-module homomorphisms. Via  classifying map $\phi_{\MGL}:\Laz\to \MGL^*(k)$, we may thus form   the tensor product ring $R^*(X):=\MGL^*(X)\otimes_{\Laz}R$.  
 
\begin{lemma} Let $k$ be a perfect field, $F\in R[\![u,v]\!]$ a formal group law over a commutative (graded) ring $R$. Assume that the exponential characteristic of $k$ is invertible in $R$. Then there is an oriented cohomology theory $R^*$ on $\Sm_k$ with $R^*(X)=\MGL^*(X)\otimes_{Laz}R$. Moreover, $R^*(k)=R$ and $R^*$ has formal group law $F$. Finally, if the characteristic of $k$ is zero, then $R^*$ is the universal oriented cohomology theory on $\Sm_k$ with formal group law $F\in R[\![u,v]\!]$.
\end{lemma}

\begin{proof} The fact that $X\mapsto R^*(X)$ extends to an  oriented cohomology theory $R^*$ on $\Sm_k$ with $R^*(X)=\MGL^*(X)\otimes_{Laz}R$ follows from the fact that $\MGL^*$ is an an  oriented cohomology theory   on $\Sm_k$, together with  standard properties of the tensor product.

By the main result of \cite{Hoyois}, the map   $\Laz_*\to \MGL_{2*,*}(k)$ is an isomorphism after inverting the exponential characteristic of $k$, and thus $R^*(k):= \MGL^*(k)\otimes_{Laz}R\cong R$ and hence the theory $R^*$ has formal group law $F\in R[\![u,v]\!]$. 

The last statement follows from the fact (remark~\ref{rem:MGLOrientedCoh}) that $\MGL^*$ is the universal oriented cohomology theory on $\Sm_k$ in case $k$  has characteristic zero.
\end{proof}

\subsection{Landweber exactness}
The well-known condition of {\em Landweber exactness} is a sufficient condition for the theory $X\mapsto  R^*(X)$ on $\Sm_k$ to arise from a motivic oriented cohomology theory.

Let  $(R,F)$ be a formal  group law.  For any prime $l>0$, we  write   $[l]\cdot_F x=\sum_{i\geq 1}a_ix^i$, and let $v_n := a_{l^n}$. In particular we have $v_0= a_1= l$.

\begin{definition}
The formal group law $(R,F)$ is said to be Landweber exact if for all primes $l$ and for all integers $n$, the multiplication map 
\[
v_n : R/(v_0,\cdots,v_{n-1})\to R/(v_0,\cdots,v_{n-1})
\]
is injective.
\end{definition}

\begin{theorem}\label{thm: Landweber}
Let $k$ be a perfect field, $(R,F)$ a formal group law. If $(R,F)$ is Landweber exact and the exponential characteristic of $k$ is invertible in $R$, then the oriented cohomology theory $X\mapsto R^*(X)$ on $\Sm_k$ is represented by  a motivic oriented cohomology theory.

More precisely, there is a motivic oriented cohomology theory $(\sE,\vartheta)$ with classifying map $\Theta_\sE:\MGL\to \sE$ and a ring homomorphism $R\to \sE^{2*,*}(k)$  such that the map of oriented cohomology theories $\oct(\Theta_\sE):\MGL^*\to \sE^*$ induces an isomorphism of oriented cohomology theories $\bar\oct(\rho_\sE):R^*:=\MGL^*\otimes_{Laz}R\to \sE^*$
\end{theorem}
The classical precursor of this result  is due to Landweber \cite{Landweber}; in our setting this result follows from  \cite{NSO}. The proof relies in addition on the main result of Appendix~\ref{App:B}, kindly supplied by Jo\"el Riou.

\begin{proof}[Proof of Theorem~\ref{thm: Landweber}] From \cite[Theorem 8.6, Proposition 8.8]{NSO} and the fact that the Tate subcategory of $\SH(\Z)$ is a Brown category \cite[Lemma 8.2]{NSO}, there is a weak commutative ring $\PP^1$ spectrum $\sE\in \SH(k)$, a morphism of  weak ring spectra $\phi:\MGL\to \sE$ and a  map  of graded rings $\rho:R_*\to \sE_{2*,*}(k)$ making the diagram
\[
\xymatrix{
\Laz_*\ar[r]^\rho\ar[d]&R_*\ar[d]\\
\MGL_{2*,*}(k)\ar[r]_\phi&\sE_{2*,*}(k)
}
\]
commute, and so that the induced map $\MGL_{**}(\sF)\otimes_{\Laz} R_*\to \sE_{**}(\sF)$ is an isomorphism for all $\sF$ in $\SH(k)$. We define $\vartheta_\sE$ to be $\phi_*(\vartheta_{\MGL})$, giving us the morphism $\phi:(\MGL,\vartheta_{\MGL})\to (\sE, \vartheta_\sE)$ of motivic oriented cohomology theories.

We claim that  the map $\MGL^{**}(X)\otimes_{\Laz} R^*\to \sE^{**}(X)$ induced by $\phi$ and $\rho$ is an isomorphism for all $X\in \Sm_k$, in particular, $\sE$ represents the oriented cohomology theory $X\mapsto R^*(X)$ on $\Sm_k$. 

For this, let $p$ denote the exponential characteristic of $k$. We have the natural  isomorphism  for $\sA\in \SH(k)$, $X\in \Sm_k$:
$\sA^{a,b}(X)[1/p]\cong \sA_{-a,-b}((\Sigma_T^\infty X_+)^D)[1/p]$,
where $(\Sigma_T^\infty X_+)^D$ is the strong dual of $\Sigma_T^\infty X_+$, which exists (in $\SH(k)[1/p]$) by corollary~\ref{cor:Duality}. As $p$ is invertible in $R^*$, this gives us the commutative diagram
\[
\xymatrix{
\MGL^{*,*}(X)\otimes_{\Laz} R^*\ar[d]_{(\phi,\rho)}\ar[r]^-\sim&\MGL_{-*,-*}((\Sigma_T^\infty X_+)^D)\otimes_{\Laz} R_{-*}\ar[d]^{(\phi,\rho)}\\
\sE^{*,*}(X)\ar[r]_-\sim&\sE_{-*,-*}((\Sigma_T^\infty X_+)^D) 
}
\]
As $(\phi,\rho):\MGL_{-*,-*}((\Sigma_T^\infty X_+)^D)\otimes_{\Laz} R_{-*}\to \sE_{-*,-*}((\Sigma_T^\infty X_+)^D)$ is an isomorphism, so is $(\phi,\rho):\MGL^{*,*}(X)\otimes_{\Laz} R^* \to \sE^{*,*}(X)$.
\end{proof}

We denote the motivic oriented cohomology theory associated to a Landweber exact formal group law $(R, F)$ by $(\MGL\otimes_{\Laz} R,\vartheta_R)$ and the canonical morphism given by the universality of $\MGL$ by $\Theta_{F,R}:(\MGL,\vartheta_{\MGL})\to (\MGL\otimes_{\Laz} R,\vartheta_R)$. We sometimes drop the orientation from the notation if the context makes the meaning clear.

\subsection{Exponential and logarithm}\label{subsec: rational_coeff}
Let $(R,F)$ be the formal group law.
A \textit{logarithm} of the formal group law $F$ is a series $g(u)=u+\sum_{i\ge2}g_iu^i\in R[\![u]\!]$ satisfying the equation \[g(F(u,v))=g(u)+g(v).\] Novikov \cite{N} showed that every formal group law with coefficients in a $\Q$-algebra has a logarithm. The functional inverse $\lambda(u)\in R[\![u]\!]$ of the logarithm $g(u)$ is called the \textit{exponential} of the formal group law.
With our grading conventions, if we give $u$ degree one, then the power series $g(u)$ and $\lambda(u)$ are both homogeneous of degree one. Thus, if we write $\lambda(u)=u+\sum_{i\ge 1}\tau_iu^{i+1}$, then $\tau_i\in R$ has degree $-i$. It is noted by Hirzebruch that ring homomorphisms $\Laz\to R_\QQ$ are in one to one correspondence with power series $\lambda$ as above. 

\subsection{Twisting a cohomology theory} We recall Quillen's twisting construction for oriented cohomology theories and its analog for motivic oriented cohomology theories. For details we refer the reader to \cite[\S4.18 and \S4.19]{LM}.

Let $A^*$ be an oriented cohomology theory on $\Sm_k$ and $\tau=(\tau_i)\in \prod_{i=0}^\infty A^{-i}(k)$, with $\tau_0=1$. Let $\lambda_\tau(u)=\sum_{i=0}^{\infty}\tau_i u^{i+1}$. The  {\em Todd genus} is the power series $\Td_\tau(t) :=t/\lambda_\tau(t)$.  

For a vector bundle $E$ on some $Y\in\Sm_k$, the Todd class of a vector bundle $E$ on $Y$ is given as usual by the splitting principal:
$\Td_\tau(E)=\prod_{i=1}^r\Td_\tau(\xi_i)\in A^*(Y)$,
where $\xi_1,\ldots, \xi_r$ are the Chern roots of $E$ in $A^*(-)$.  The assignment $E\mapsto \Td_\tau(E)$ is multiplicative in exact sequences, hence descends to a well-defined homomorphism $\Td_\tau:K_0(Y)\to (1+A^{*\ge 1}(Y))^\times$.

The twisted oriented cohomology theory $A^*_\tau$  is defined by setting  $A^*_\tau(X)=A^*(X)$, for $X\in \Sm_k$, and for $f:Y\to X$, setting  $f_\tau^*=f^*:A^*(X)\to A^*(Y)$.
For a projective morphism $f:Y\to X$ and $\alpha\in A^*(Y)$, we set
\[
f_*^\tau(\alpha):=f_*(\Td_\tau( [T_Y]-[f^*T_X])\cdot \alpha).
\]
One computes that $c_1^{A_\tau}(L)=\lambda_\tau(c_1^A(L))$ and that the formal group law for $A^*_\tau$ is given by
\[
F_A^\tau(u, v)=\lambda_\tau(F_A(\lambda_\tau^{-1}(u), \lambda_\tau^{-1}(v))).
\]
Thus,  if $F_A$ is the additive group law, $F_A(u,v)=u+v$, then $\lambda_\tau(u)$ is the exponential map for the twisted group law $(F_A^\tau, A^*(k))$.

The twisting construction is also available for motivic oriented cohomology theories. Let $(\sE, \mu, 1, \vartheta)$ be a motivic oriented cohomology theory, and let $(\tau_i\in \sE^{-2i, -i}(k))_{i\ge0}$ be a sequence of elements with $\tau_0=1$.

Form the orientation $\vartheta_\tau\in \sE^{2,1}(\PP^\infty/0)$ by setting $\vartheta_\tau:=\lambda_\tau(\vartheta)=\sum_{i\ge0}\tau_i\vartheta^{i+1}$. The projective bundle formula  implies that $\vartheta^m$ goes to zero in $\sE^{2,1}(\PP^N/0)$ for $m>N$, from which it follows that
$\lambda_\tau(\vartheta)$ is a well-defined element in $\sE^{2,1}(\PP^\infty/0)=\underleftarrow{\lim}\sE^{2,1}(\PP^N/0)$ and is an orientation. Let ${}_\tau\sE^*$ be the oriented cohomology theory corresponding to $(\sE, \mu, 1, \vartheta_\tau)$.

It follows  from the definitions and \cite[Corollary 1.1.10 ]{PaninRR}  that the identity map on the graded groups $\sE^*_\tau(X)=\sE^*(X)$ defines an isomorphism of oriented cohomology theories 
\[
\oct((\sE, \mu, 1, \vartheta_\tau))\to \sE^*_\tau:=\oct((\sE, \mu, 1, \vartheta))_\tau; 
\]
we will henceforth identify these two oriented cohomology theories via this isomorphism and write $\sE^*_\tau$ for both $\oct((\sE, \mu, 1, \vartheta_\tau))$ and $\oct((\sE, \mu, 1, \vartheta))_\tau$.

Let  $H\Z\in \SH(k)$ be the commutative $\PP^1$ ring spectrum representing integral motivic cohomology \cite{VoevICM} (see also \cite[\S2.4]{RO}, where this spectrum is denoted $M\ZZ$). $H\Z$ has the  orientation $v_H\in H\Z^{2,1}(\PP^\infty/0)$ is given by the sequence of hyperplane classes $c_1^{\CH}(\sO_{\PP^n}(1))\in H\Z^{2,1}(\PP^n)=\CH^1(\PP^n)$.

 Let  $R$ be a $\Z$-graded commutative ring. We have the  motivic oriented cohomology theory on $\Sm_k$ 
\[
HR=\oplus_{n\in\Z}\Sigma_{\PP^1}^n HR^{-n},
\]
where $HR^{-n}$ is the $\PP^1$ spectrum representing motivic cohomology with coefficients $R^{-n}$ and where we give $HR$ the   the orientation induced by that of $H\Z$.  

Now let $(R, F)$ be a formal group law, with $R$ a (graded) $\Q$-algebra. Taking $(\tau_i\in R^{-i}=HR^{-2i,-i}(k))_i$ to be the sequence such that $\lambda_\tau(u)$ is the exponential function for $( R, F)$, we form the twisted motivic oriented cohomology theory $HR_\tau$; by construction $HR_\tau$ has associated formal group law $(R, F)$. We note that $(R, F)$ is Landweber exact, since $R$ is a $\Q$-algebra.

We may also form the motivic oriented cohomology theory $\MGL\otimes_{\Laz} R$ associated to the Landweber exact formal group law $(R ,F )$. As this is the universal motivic  oriented cohomology theory with group law $(R, F)$, the classifying map $\Theta_{HR_\tau}:\MGL\to HR_\tau$ factors through $\Theta_{F,R}:\MGL\to \MGL\otimes_{\Laz} R$, giving the induced classifying map $\bar{\Theta}_{HR_\tau}:\MGL\otimes_{\Laz} R\to HR_\tau$, unique up to a phantom map.

\begin{lemma} \label{lem:MotReal} The map $\bar{\Theta}_{HR_\tau}:\MGL\otimes_{\Laz} R\to HR_\tau$ induces an isomorphism 
$\MGL^{*,*}\otimes_{\Laz} R\to HR^{*,*}_\tau$ of bi-graded cohomology theories on $\Sm_k$, in particular, we have the isomorphism of associated oriented cohomology theories on $\Sm_k$
\[
\bar{\Theta}_{HR_\tau}:R^*=\MGL^*\otimes_{\Laz} R\to HR^*_\tau
\]
\end{lemma}

\begin{proof} We apply the slice spectral sequence to $\MGL\otimes_{\Laz} R$ and $HR_\tau$; the map $\bar{\Theta}_{HR_\tau}$ induces a map of spectral sequences. For a $\PP^1$ spectrum $\sE$ denote the $n$th layer in the slice tower for $\sE$ as $s_n\sE$. For an abelian group $A$, it follows from  \cite[Theorem 6.5.1, Theorem 9.0.3]{Lev_Coniveau} that
\[
s_n HA=\begin{cases} HA&\text{ for }n=0\\0&\text{ for }n\neq0.\end{cases}
\]
Also $s_n\Sigma^m_{\PP^1}\sE=\Sigma^m_{\PP^1}s_{n-m}\sE$. As $HR=\oplus_{n\in\Z}\Sigma_{\PP^1}^n HR^{-n}$, it follows  that the $n$th layer in the slice tower for $HR$ is given by  $s_nHR=\Sigma_{\PP^1}^n HR^{-n}$. By Spitzweck's computation of the layers in the slice tower for a Landweber exact theory (\cite[Theorem 6.1]{Spitzweck}, relying on \cite{Hoyois}), we have the same $s_n\MGL\otimes_{\Laz} R=\Sigma_{\PP^1}^n HR^{-n}$ as well. The maps on the layers of the slice tower induced by $\bar{\Theta}_{HR_\tau}$ are $H\Q$-module maps (by results of Pelaez \cite[Theorem 3.6.14]{Pelaez}), and one knows by a result of Cisinski-Deglise \cite[Theorem 16.1.4]{CD} that
\[
\Hom_{H\Q-\text{Mod}}(\Sigma_{\PP^1}^n HR^{-n}, \Sigma_{\PP^1}^n HR^{-n})\cong \Hom_{\QQ-\text{Vec}}(R^{-n}, R^{-n})
\]
In particular, the map $s_n\bar{\Theta}_{HR_\tau}:\Sigma_{\PP^1}^n HR^{-n}\to \Sigma_{\PP^1}^n HR^{-n}$ is determined by the induced map after applying the functor $H^{-2n,-n}(k,-)$, that is, on the coefficient rings of the theories $\MGL^{*,*}\otimes_{\Laz} R$ and $HR^{*,*}_\tau$. However, by construction, this is the map $\bar{\Theta}:R\to R$ induced by the classifying map $\Laz\to R$ associated to the formal group $F_\tau(u,v)=\lambda_\tau(g_\tau(u)+g_\tau(v))$. As this latter formal group law is equal to $F$ by construction, the map $\bar{\Theta}:R\to R$ is the identity map.

Thus ${\Theta}_{HR_\tau}$ induces an isomorphism of the slice spectral sequences. As these are bounded and convergent \cite[Theorem 8.12]{Hoyois}, $\Theta_{HR_\tau}$ is an isomorphism of bi-graded cohomology theories on $\Sm_k$.
\end{proof}

\subsection{Universality} We have already noted (remark~\ref{rem:MGLOrientedCoh}(3)) that for a field of characteristic zero, $\MGL^*$ is the universal oriented cohomology theory on $\Sm_k$. In this subsection, we use the universality of $\MGL$ as a motivic oriented cohomology theory plus some tricks with formal group laws to show that $\MGL_\QQ^*$ is the universal oriented cohomology theory for theories  in $\Q$-algebras on $\Sm_k$, for an arbitrary field $k$.

\begin{lemma}\label{lem:CHUniv} Let $k$ be a field. If $k$ has characteristic zero, then $\CH^*$ is the universal oriented cohomology theory on $\Sm_k$ with formal group law $(u+v, \Z)$. If $k$ has characteristic $p>0$, then $\CH^*_\QQ$ is the universal oriented cohomology theory on $\Sm_k$  with formal group law $(u+v, \QQ)$.
\end{lemma}

One would expect that over an arbitrary field, $\CH^*$ is the universal oriented cohomology theory on $\Sm_k$ with formal group law $(u+v, \Z)$. This does not seem to be known.

\begin{proof} The case of characteristic zero is proven in \cite[Theorem~1.2.2]{LM}. In characteristic $p>0$, let $A^*$ be an oriented cohomology theory with additive formal group law $F_A(u,v)=u+v$ and with $A^*(k)$ a $\QQ$-algebra. Extend the coefficients in the theory $A^*$ by a Laurent polynomial ring, forming the theory $A^*[t,t^{-1}]$, with $t$ of degree -1. Then take the twist with respect to the modified exponential function $\lambda_t(u):=t^{-1}(1-e^{-tu})$,
that is, $\tau_i:=(-1)^it^i/(i+1)!$. A simple computation shows that theory $A^*[t,t^{-1}]_\tau$ has the multiplicative group law $F(u,v)=u+v-tuv$, and that the twisted first Chern class is given by  $c_1^t(L)=t^{-1}(1-e^{-tc_1^A(L)})$.

One can define the modified Chern character
\[
\text{ch}^A_t:K_0[t,t^{-1}]_\Q\to A^*[t,t^{-1}]_\tau,
\]
which sends a vector bundle $E$ of rank $r$  to $\text{ch}^A_t(E):=r-tc_1^t(E^\vee)$.
For a line bundle $L$,  we have $\text{ch}^A_t(L)=1-tc_1^t(L^\vee)=e^{-tc_1^A(L^\vee))}=e^{tc_1^A(L))}$.
Using the splitting principle, the fact that $A^*$ has the  additive formal group law implies that $\text{ch}^A_t$ is a natural transformation of functors to graded $\QQ[t, t^{-1}]$-algebras. Since $c_1^{K_0[t, t^{-1}]}(L)=t^{-1}(1-L^{-1})$,  we have
\[
\text{ch}^A_t(c_1^K(L))=c_1^t(L)
\]
for all line bundles $L$. By Panin's Riemann-Roch theorem \cite[Corollary 1.1.10]{PaninRR}, this shows that $\text{ch}^A_t$ is a natural transformation of oriented cohomology theories.

We have the  Adams operations $\psi_k$, $k=1, 2, \ldots$,  on $K_0(X)$, which we extend to Adams operations on $K_0(X)[t,t^{-1}]_\QQ$ by $\QQ[t,t^{-1}]$-linearity. Define the operation $\psi^A_k$ on $A^*(X)[t,t^{-1}]$ to be the $\QQ[t,t^{-1}]$-linear map which is multiplication by $k^n$ on $A^n(X)$; it is easy to see that $\psi^A_k$ is a natural $\QQ[t,t^{-1}]$-algebra homomorphism. As $A$ has the additive group law, $c_1^A(L^{\otimes k})=kc_1^A(L)$ and thus
\[
\text{ch}^A_t(\psi_k(L))=\text{ch}^A_t(L^{\otimes k})
=e^{tc_1^A(L^{\otimes k})}
=e^{ktc_1^A(L)}
=\psi^A_k(e^{tc_1^A(L)})
=\psi^A_k(\text{ch}^A_t(L))
\]
for all line bundles $L$. By the splitting principle, this gives the identity
\[
\text{ch}^A_t\circ \psi_k=\psi^A_k\circ \text{ch}^A_t.
\]

If we take $A^*=\CH^*_\QQ$,  $\text{ch}^A_t$ is a modified version of the classical Chern character; thus by Grothendieck's classical result, the natural transformation
\[
\text{ch}^{\CH^*_\Q}_t:K_0[t,t^{-1}]_\Q\to \CH^*_\QQ[t,t^{-1}]_\tau
\]
is an isomorphism. This gives us the natural transformation of oriented cohomology theories
\[
\text{ch}^A_t\circ(\text{ch}^{\CH^*_\Q}_t)^{-1}:\CH^*_\QQ[t,t^{-1}]_\tau\to A^*[t,t^{-1}]_\tau
\]
Twisting back  gives us the natural transformation of oriented cohomology theories
\[
\vartheta^{\CH t}_A:\CH^*_\QQ[t,t^{-1}]\to A^*[t,t^{-1}].
\]

We give $\CH^*_\QQ[t,t^{-1}]$ a bi-grading by putting $\CH^n_\QQ\cdot t^m$ in bi-degree $(n,m)$, and do the same for
$A^*[t,t^{-1}]$.
Since both $\text{ch}^{\CH^*_\Q}_t$ and $\text{ch}^A_t$ commute with the Adams operations, we have
$\vartheta^{\CH t}_A\circ \psi_k^{\CH}=\psi_k^A\circ \vartheta^{\CH}_A$
and thus $\vartheta^{\CH t}_A$ respects the bi-grading. Passing to the respective quotients by the ideal $(t-1)$ gives us the natural transformation of oriented cohomology theories
\[
\vartheta^{\CH}_A:\CH^*_\QQ\to A^*
\]
where we now use the original grading on $\CH^*_\QQ$ and $A^*$.

The uniqueness of $\vartheta^{\CH}_A$ follows from  Grothendieck-Riemann-Roch. Indeed, as a natural transformation of oriented cohomology theories, $\vartheta^{\CH}_A(c_n^{\CH}(E))=c_n^A(E)$ for all  vector bundles $E$ on $X\in \Sm_k$, and all $n$. But for irreducible $X\in \Sm_k$, the Grothendieck-Riemann-Roch theorem implies that $\CH^*(X)_\QQ$ is generated as a $\Q$-vector space by the elements of the form $c_n^{\CH}(E)$, $E$ a vector bundle on $X$, $n\ge1$ an integer, together with the identity element $1\in \CH^0(X)$. Thus $\vartheta^{\CH}_A$ is unique.
\end{proof}

We now consider the generic twist $\CH^*_\Q[\mathbf{b}]_{\mathbf{b}}$, where $\mathbf{b}=\{1=b_0, b_1, b_2,\ldots\}$. We have as well the motivic oriented cohomology theory $H\QQ\in \SH(k)$. We may form the generic twist $H\QQ[\mathbf{b}]_{\mathbf{b}}$ by taking the orientation $\lambda_b(v_H):=\sum_{n\ge0}b_nv_H^{n+1}$.

\begin{prop} Let $k$ be a field.
\begin{enumerate}
 \item Suppose  $k$  has characteristic zero. Then $\MGL^*$ is the universal oriented cohomology theory on $\Sm_k$.
 \item For $k$ arbitrary,   $\MGL^*_\QQ$ is the universal oriented cohomology theory in $\QQ$-algebras on $\Sm_k$.
\item  Let $(F, R)$ be a formal group law, with $R$ a $\Q$-algebra and let $\lambda_\tau(u)=\sum_{n\ge0}\tau_nu^{n+1}$ be the associated exponential function. Then the classifying map $\MGL_\QQ^*\otimes_{\Laz} R\to (\CH^*\otimes R)_\tau$ is an isomorphism.
\item Consider the generic twist $\CH^*_\QQ[\mathbf{b}]_{\mathbf{b}}$. The classifying map $\MGL^*_\QQ\to \CH^*_\QQ[\mathbf{b}]_{\mathbf{b}}$ is an isomorphism.
\end{enumerate}
\end{prop}

\begin{proof} (1) is just remark~\ref{rem:MGLOrientedCoh}~(3).

(4) follows from (3), noting that the classifying map $\Laz\to \Z[b_1, b_2,\ldots]$  for the formal group law   $F_b(u,v):=\lambda_b(\lambda_b^{-1}(u)+\lambda_b^{-1}(v))$ induces an isomorphism
$\Laz_\QQ\to \QQ[b_1, b_2,\ldots]$ (see e.g. \cite[Theorem 7.8]{Adams}), hence $\MGL_\QQ^*\to \MGL_\QQ^*\otimes_{\Laz_\QQ}\QQ[\mathbf{b}]$ is an isomorphism. The assertion (3) follows immediately from Lemma~\ref{lem:MotReal},  as the isomorphism $\CH^*\cong H^{2*,*}(-,\Z)$ gives rise to a canonical isomorphism $HR^*_\tau\cong (\CH^*\otimes R)_\tau$.

To prove (2), as we have inverted the characteristic, we may replace $k$ with its perfect closure, so we may assume that $k$ is perfect.  It suffices by (4) to show that $\CH^*_\QQ[\mathbf{b}]_{\mathbf{b}}$ is the universal oriented cohomology theory in $\Q$-algebras on $\Sm_k$. This follows by applying the twisting construction. Indeed,  let $A^*$ be an oriented cohomology theory on $\Sm_k$, such that $A^*(k)$ is a $\Q$-algebra.  Let $(\tau_n^{-1}\in A^{-n})_n$ be the sequence such that $\lambda_{\tau^{-1}}(u)$ is the logarithm of the formal group law $(F_{A^*}, A^*(k))$. The twist $A^*_{\tau^{-1}}$ thus has the additive formal group law and hence by Lemma~\ref{lem:CHUniv} we have the (unique) classifying map $\theta_{\tau^{-1}}:\CH^*_\QQ\to A^*_{\tau^{-1}}$. As we have already mentioned above, the map $\Laz_\QQ\to \QQ[b_1, b_2, \ldots]$ classifying the formal group law  $F_b(u,v)$ is an isomorphism, hence there is a unique ring homomorphism $\phi: \QQ[b_1, b_2, \ldots]\to A^*(k)$ with $\phi(F_b(u,v))=F_{A^*}(u,v)$.   Extend $\theta_{\tau^{-1}}$ to $\Theta_{\tau^{-1}}:\CH^*_\QQ[\mathbf{b}]\to A^*_{\tau^{-1}}$ by using $\phi$, and then twist back by $\mathbf{b}$ and $\phi(\mathbf{b})=\tau$ to get the map $\Theta_A:\CH^*_\QQ[\mathbf{b}]_{\mathbf{b}}\to A^*$ of oriented cohomology theories on $\Sm_k$. The uniqueness of $\Theta_A$ follows from the uniqueness of $\theta_{\tau^{-1}}$ and that of $\phi$. This completes the proof.
\end{proof}

Consider a formal group law $(R,F)$ with exponential $\lambda(t):=t+\sum_{i\ge 1}\tau_it^{i+1}\in R_\QQ[\![t]\!]$, $\tau_i\in R^{-i}_\QQ$, $\tau_0=1$.  We have two methods of constructing an oriented cohomology theory with formal group law $(R_\QQ,F)$: the specialisation construction $\MGL^*\otimes_{\Laz}R_\QQ$  and the twisting construction $(\CH^*\otimes R_\QQ)_\tau$.  These two oriented cohomology theories are canonically isomorphic; we will denote both of them by $R_\QQ^*$:
$\MGL^*\otimes_{\Laz}R_\QQ=:R^*_\QQ:=(\CH^*\otimes R_\QQ)_\tau$.

\section{Algebraic elliptic cohomology}\label{sec: Landweber}

\subsection{The elliptic formal group law}\label{sec:ellip group law}
An \textit{algebraic elliptic cohomology theory} is the cohomology theory corresponding to an elliptic formal group law. More precisely, let $R$ be a ring and let $p:E\to \Spec R$ be an elliptic curve over $R$  with  a  chosen local uniformizer $t$ in a neighborhood of the identity section. The expansion of the group law of $E$ in terms of the coordinate $t$ gives a formal group law $F_E$ with coefficients in $R$.

There is a well-studied elliptic formal group law and the corresponding cohomology theory. Let $A=\ZZ[\mu_1, \mu_2, \mu_3, \mu_4, \mu_6]$. We take $E$ to be the Weierstrass curve \[y^2 + \mu_1xy + \mu_3y = x^3 + \mu_2x^2 + \mu_4x + \mu_6\] over the ring $A$ and use $t=y/x$ as the local uniformizer. This formal group law will be referred as the TMF elliptic formal group law. See \cite{Hop_ICM} and \cite{H_MITNotes} for survey of this theory. The classifying map $F_E :\Laz\to A[\Delta^{-1}]$   is Landweber exact, where $\Delta$ is the discriminant of $E_A/A$.

The elliptic formal group law we are using is called the Krichever elliptic formal group law in the literature. It is related to, but different from the TMF elliptic formal group law. We recall the genus corresponding to this formal group law, following the convention in \cite{T}.\footnote{Krichever uses the function $te^{-k_0t}\Phi(t, z;\tau)$ as does H\"ohn, except that H\"ohn leaves the exponential factor $\text{exp}(-k_0t+\zeta(z,\tau)t)$ as an unspecified ``constant of integration''. Totaro introduces the change of variable $t\mapsto \frac{t}{2\pi i}$ as we do, so that the resulting function is expressible in terms of $e^x$ instead of $e^{2\pi i x}$. This has the effect of replacing H\"ohn's choice of lattice $\<2\pi i, 2\pi i\tau\>$ with our $\<1,\tau\>$. Totaro's $k$ is different from ours, but this only affects formulas for bundles with non-zero first Chern class. In particular, the rigidity property holds for all these elliptic theories.} 

Let
\[
\sigma(z, \tau):=z\prod_{w\in \Z+\Z\tau, w\neq 0}(1-\frac{z}{w})e^{\frac{z}{w}+\frac{1}{2}(\frac{z}{w})^2}
\]
be the Weierstrass sigma function. Set $\zeta(z,\tau):=\partial \log\sigma(z,\tau)/\partial z$. The {\em Baker-Akhiezer function} is defined as
\[\Phi(t, z;\tau)=e^{\zeta(z,\tau)t}\frac{\sigma(z-t, \tau)}{\sigma(t, \tau)\sigma(z, \tau)}\]
and we let
\begin{equation}\label{eqn:EllipticExp}
Q(t):=\frac{t}{2\pi i}e^{kt}\Phi(\frac{t}{2\pi i}, z;\tau), \quad \lambda(t):=\frac{t}{Q(t)}.
\end{equation}
The coefficients of the power series $Q(t)$ are in $\QQ(\!(e^{2\pi iz})\!)[\![e^{2i\pi\tau},k]\!]$ (see e.g. \cite[pg. 765]{T}).

We define the algebraic elliptic genus as the ring homomorphism
\[
\phi_E: \Laz\to \QQ(\!(e^{2\pi iz})\!)[\![e^{2i\pi\tau},k]\!]
\]
associated to the Hirzebruch genus $Q(t)$, that is, the associated formal group law has exponential   power series $\lambda(t)$.   Explicitly,  the elliptic genus of a $n$-dimensional smooth variety $X$ is defined by
$\phi(X):=\langle\prod_{i=1}^nQ(\xi_i), [X]\rangle$, where the $\xi_i$ are the Chern roots of the tangent bundle $\TT_X$ of $X$.

The coefficient ring $\tilde\Ell$ of the   the elliptic formal group law is by definition the image of $\phi_E$ and the elliptic formal group law is by definition
\[
x+_E y=\lambda(\lambda^{-1}(x)+\lambda^{-1}(y)).
\]
The algebraic elliptic cohomology theory associated to this formal group law, $\MGL^*\otimes_{\Laz}\tilde\Ell$, is denoted by $\tilde\Ell^*$.

Define functions $a_2, a_3, a_4$ by
\begin{equation}\label{eqn:Generators}
a_2 =(2\pi i)^{-2}{\wp}(z;\tau), a_3=(2\pi i)^{-3}{\wp}'(z;\tau), a_4=\frac{1}{2}(2\pi i)^{-4}g_2(\tau).
\end{equation}
where ${\wp}(z;\tau)$ is the Weierstrass ${\wp}$-function and $g_2(\tau)$ is the  Eisenstein series $60G_4(\tau)$.   Let $a_1=k\in \QQ(\!(e^{2\pi iz})\!)[\![e^{2i\pi\tau},k]\!]$.

Recalling the classical expansions of the functions  $a_2, a_3, a_4$ in terms of $e^{2i\pi\tau}$ and $e^{2\pi iz}$ (see e.g. \cite[pg. 46]{Lang}) defines these functions as elements of $\QQ(\!(e^{2\pi iz})\!)[\![e^{2i\pi\tau},k]\!]$.  We let $R=\Z[a_1, a_2, a_3, \frac{1}{2}a_4]\subset \QQ(\!(e^{2\pi iz})\!)[\![e^{2i\pi\tau},k]\!]$ be the corresponding  subalgebra of $\QQ(\!(e^{2\pi iz})\!)[\![e^{2i\pi\tau},k]\!]$.

As the general Weierstrass equation is of the form
\[
y^2=4x^3-g_2(\tau)x -g_3(\tau)
\]
and is parametrized by $(x,y)=({\wp}(z,\tau), {\wp}'(z,\tau))$, the functions ${\wp}(z,\tau), {\wp}'(z,\tau), g_2(\tau)$ are algebraically independent in $\QQ(\!(e^{2\pi iz})\!)[\![e^{2i\pi\tau},k]\!]$ over $\QQ$, and thus $R$ is a polynomial ring over $\Z$. Giving $a_i$ degree $-i$ gives $R$ a grading.

 It follows directly from  \cite[Lemma~44]{BB} that $\tilde{\Ell}$ is contained in the subring $R$ and that $\phi_E:\Laz
 \to R$ is a homomorphism of graded rings.\footnote{See \cite[\S5.1]{BB}. Again, we have a different normalization from that of \cite{BB}, where they use the function $\Phi(t, z;\tau)$ instead of $\Phi(\frac{t}{2\pi i}, z;\tau)$. This accounts for the factors of $2\pi i$.} This is accomplished as follows. We let $E_R\to\Spec R$ be the elliptic curve   over $R$ defined as the base change from the Weierstrass curve on $A=\ZZ[\mu_1,\mu_2,\mu_3,\mu_4,\mu_6]$ via the map of rings $\varphi: A\to R$
\[
\mu_1\mapsto2a_1,\ \mu_2\mapsto 3a_2-a_1^2,\
\mu_3\mapsto-a_3,\
\mu_4\mapsto -\frac{1}{2}a_4+3a_2^2-a_1a_3,\
\mu_6\mapsto0.
\]
The chosen parameter $t$ on the Weierstrass curve gives by base-change a parameter $t_R$ along the zero section of $E_R$. This gives us a formal group law over $R$, this just being the one induced from the TMF formal group law through change of coefficient ring via the map $\varphi$. An explicit isomorphism of this formal group law with Krichever's elliptic formal group law (after extending the coefficient ring from $\tilde{\Ell}$ to $R$) is given in \cite[Lemma~44]{BB}.

\begin{remark} We shall see below that $\tilde\Ell_\Q=\Q[a_1,a_2,a_3,a_4]$. However, as noted by H\"ohn \cite{H} and Totaro \cite[\S6]{T}, $\tilde\Ell$ itself is not finitely generated over $\Z$. 
\end{remark}

Let $\Ell=\ZZ[a_1,a_2,a_3,\frac{1}{2}a_4][\Delta^{-1}]$, where $\Delta$ is the discriminant of $E_R/R$. An explicit formula for $\Delta$ is
\begin{align}\label{discr}
\Delta&=36(-4a_1a_3-a_4+6a_2^2)^2 a_2^2-8(-4a_1a_3-a_4+6a_2^2)^3\\&\hskip30pt-27a_3^4
+108(-4a_1a_3-a_4+6a_2^2) a_2 a_3^2-432a_2^3 a_3^2.\notag
\end{align}

\begin{theorem}[Theorem~\ref{thm: ellip_exact_intro}]\label{thm: ellip_exact}
The Krichever elliptic formal group law $(\Ell[1/2],F_{Kr})$ is Landweber exact. Therefore, for $k$ a field of exponential characteristic $p$,  the oriented cohomology theory  $\Ell[1/2p]^*:=\MGL^*\otimes_{\Laz}\Ell[1/2p]^*$ on $\Sm_k$  is represented by a motivic oriented cohomology theory $\sEll[1/2p]$ on $\Sm_k$.
\end{theorem}
The second statement follows from the first by Theorem~\ref{thm: Landweber}.  The Landweber exactness of $(\Ell[1/2],F_{Kr})$  can be proved following the proof of the Landweber exactness of the TMF formal group law. As a sketch, we mention that the injectivity of $v_i$ for $i>0$ is related to the height of the formal group law of these curves and that the only possible height of these curves are 1, 2, or infinity. The generic members in our family all have height one, which shows that multiplication by $v_1$ is injective.  Finally, we claim that  $v_2$ in $R_l[\Delta^{-1}]/(v_1)$ is a unit. This claim implies
that multiplication by $v_2$ is injective and that  $R_l[\Delta^{-1}]/(v_1,v_2)=0$, which implies the required injectivity for $n\geq 3$. To verify the claim, assume otherwise, then $v_2$ is contained in a maximal ideal $\mathfrak{m} \subset R_l[\Delta^{-1}]/(v_1)$. Therefore, the fiber of the family of curves over this closed point has associated formal group law with height greater than $2$. This
contradicts with the fact that the height of the formal group law of elliptic curves over a field
of characteristic $l > 0$ can only be 1 or 2.

\section{Flops in the cobordism ring}\label{sec:CobordFlop}

\subsection{Double point cobordism}

Following \cite{LP}, we  have the double-point cobordism $\omega^*(X)$ of $X\in \Sm_k$, defined as the graded abelian group $M^*_+(X)$  generated by  projective morphisms $Y\to X$, with $Y\in \Sm_k$ irreducible, modulo the  \textit{double point relation} of \cite[Definitions 0.1, 0.2]{LP}. It is shown in \cite[Theorem 1]{LP} that if $k$ has characteristic zero, then $\omega^*(X)$ is  canonically isomorphic to  $\Omega^*(X)$. We will use a weaker version of this theorem, which holds in a characteristic free fashion.

\begin{prop}[Levine and Pandharipande, \hbox{\cite[Proposition 3.5]{LM}}]
For any field $k$ and scheme $X\in \Sm_k$ of finite type over $k$, the natural projection $\Pi:M^*_+(X)\to \Omega^*(X)$ factors through $\omega^*(X)$.
\end{prop}
The proof in \cite{LP} uses only the existence of smooth pull-back, projective push-forward, the first Chern class of a line bundle, and external product, which means it does not depend on any assumption on $k$. Thus, the double point relation also holds when the field $k$ has positive characteristic.

We note that $\omega^*$ has the following structures:
\begin{enumerate}
\item pullback maps $f^*:\omega^*(X)\to \omega^*(Y)$ for each smooth morphism $f:Y\to X$ in $\Sm_k$.
\item push-forward maps $f_*:\omega^*(Y)\to \omega^{*-d}(X)$ for each projective morphism $f:Y\to X$ of relative dimension $d$ in $\Sm_k$.
\item associative, commutative external products, and an identity element $1\in\omega^0(k)$.
\end{enumerate}
The pullback and pushforward maps are functorial, and  are compatible with the external products.

Composing the map $\omega^*\to \Omega^*$ with the natural transformation $\Theta_{\MGL}$ gives us the  transformation
\[
\theta_{\MGL}:\omega^*\to \MGL^*,
\]
natural with respect to smooth pullback, projective push-forward, external products and unit.

Let $F \subseteq X$ be a smooth closed subscheme of some $X\in \Sm_k$. The double point relation yields the following blow-up formula in
$\omega^*(X)$, and hence in $\Omega^*(X)$ and $\MGL^*(X)$:
\begin{equation}\label{eqn: blow_up0}
1_X = [Bl_F X\to X]+[\PP(N_F X \oplus \sO)\to X] -[\PP_{\PP(N_FX)}(\sO_{N_F X}(1) \oplus \sO)\to X]
\end{equation}
where $Bl_F X$ is blow-up of $X$ along $F$, and $N_F X$ is the normal bundle of $F$. This is proved by the usual method of   deformation to the normal cone. In case $X$ is projective over $k$, pushing forward to $\Spec k$ gives the relation in $\omega^*(k)$,  $\Omega^*(k)$ and $\MGL^*(k)$
\begin{equation}\label{eqn: blow_up}
[X] = [Bl_F X]+[\PP(N_F X \oplus \sO)] -[\PP_{\PP(N_FX)}(\sO_{N_F X}(1) \oplus \sO)]
\end{equation}

We say two smooth projective $n$-folds $X_1$ and $X_2$ are related by a flop if we have the following diagram of projective birational morphisms:
\begin{equation}
\xymatrix{
& \widetilde{X}\ar[ld]\ar[rd] & \\
X_1\ar[rd]^{p_1}& & X_2\ar[ld]_{p_2}\\
& Y & \\
}
\end{equation}
Here $Y$ is a singular projective $n$-fold with singular locus $Z$, such that $Z$ is smooth of dimension $n-2k+1$. We assume in addition that there exist rank $k$ vector bundles $A$ and $B$ on $Z$, such that the exceptional locus $F_1$ in $X_1$ is the $\PP^{k-1}$-bundle $\PP(A)$ over $Z$, with normal bundle $N_{F_1}X_1=B\otimes \sO_A(-1)$. Similarly, the the exceptional locus $F_2$ in $X_2$ is $\PP(B)$, with normal bundle $N_{F_2}X_2=A\otimes \sO_B(-1)$.
Let $Q^3\subset \PP^4$ denote the 3-dimensional quadric with an ordinary double point $v$, defined by the equation $x_1x_2=x_3x_4$.  We say that $X_1$ and $X_2$ are related by a {\em classical} flop if in addition  $k=2$, and  along $Z$, $(Y,Z)$ is {\em Zariski} locally isomorphic to
 $(Q^3\times Z, v\times Z)$.


We assume now $X_1$ and $X_2$ are related by a   flop.
Let $X=X_1$ and $F=F_1$, the terms on the right hand side of formula \eqref{eqn: blow_up} become:
\begin{align*}
&Bl_F X=\widetilde{X},\\
&\PP(N_F X \oplus \sO)=\PP_{\PP(A)} (B\otimes \sO_A(-1) \oplus \sO),\\
&\PP_{\PP(N_X)}(\sO_{N_F X}(1) \oplus \sO)=\PP_{\PP(B\otimes \sO_A(-1))}(\sO_{B\otimes \sO_A(-1)}(1) \oplus \sO),
\end{align*}
where $\PP_{\PP(A)} (B\otimes \sO_A(-1) \oplus \sO)$ is a projective bundle over $\PP(A)$,  which in turn is a projective bundle over $Z$; and $\PP_{\PP(B\otimes\sO_A(-1))}(\sO_{B\otimes \sO_A(-1)}(1) \oplus \sO)$ is a projective bundle over $\PP(B\otimes\sO_A(-1))$,  which is a projective bundle over $\PP(A)$.

Thus, we get the following immediate lemma.
\begin{lemma}\label{lem: double_pt}
In the cobordism ring $\MGL^*(k)$, we have
\begin{align*}
[X_1]-[X_2]=&[\PP_{\PP(A)} (B\otimes \sO_A(-1) \oplus \sO)]-[\PP_{\PP(B\otimes\sO_A(-1))}(\sO_{B\otimes \sO(-1)}(1) \oplus \sO)]\\
&-[\PP_{\PP(B)} (A\otimes \sO_B(-1) \oplus \sO)]+[\PP_{\PP(A\otimes\sO_B(-1))}(\sO_{A\otimes \sO(-1)}(1) \oplus \sO)].
\end{align*}
In particular, $[X_1]-[X_2]$ in $\MGL^*(k)$ comes from an element in $\MGL^*(Z)$.
\end{lemma}
The second claim follows from the observation that each of the projective bundles are smooth varieties over $Z$. We will abuse notation by denoting a lifting of $[X_1]-[X_2]$ to $\MGL^*(Z)$ by $[X_1]-[X_2]$ itself.

\subsection{Flops in the cobordism ring}\label{sec: flop in cobord}
Since each term in the formula in Lemma~\ref{lem: double_pt} is a iterated projective bundle over $Z$, we will apply Quillen's formula iteratively to each term, to calculate the fundamental class of the iterated projective bundles in $\MGL^*(Z)$.

\begin{prop}\label{prop:Diff_Omega}
Let $a_1,\dots, a_k$ be the Chern roots of the bundle $A$ over $Z$ and let $b_1, \dots,b_k$ be the Chern roots of the bundle $B$ over $Z$, all Chern roots to be taken in $\MGL^*$. Then in $\MGL^*(Z)$, we have
\begin{equation*}
[X_1]-[X_2]=\sum_{m=1}^k\left(\frac{1}{\prod_{i=1}^k(b_i+_\Omega a_m)\prod_{l\neq m}(a_l-_\Omega a_m)}-\frac{1}{\prod_{i=1}^k(a_i+_\Omega b_m)\prod_{l\neq m}(b_l-_\Omega b_m)}\right).
\end{equation*}
\end{prop}

The rest of this subsection is devoted to prove this proposition. We let $u_B:=
c_1(\sO_{B\otimes \sO_A(-1)}(1))$, $v_A:=c_1(\sO_A(1))$, $\PP^B:=\PP_{\PP(A)}(B\otimes \sO_A(-1))$,
$\PP^A:=\PP_{\PP(B)}(A\otimes \sO_B(-1))$.

\subsubsection{The term $\PP_{\PP^B}(\sO_{B\otimes \sO_A(-1)}(1) \oplus \sO)$}
We first prove the following
\begin{lemma}
We have $[\PP_{\PP^B}(\sO_{B\otimes \sO_A(-1)}(1) \oplus \sO)]=[\PP_{\PP^A}(\sO_{A\otimes \sO_B(-1)}(1)\oplus \sO)]$ in $\MGL^*(Z)$.
\end{lemma}

\begin{proof}
Let $\pi_1: \PP_{\PP^B}(\sO_{B\otimes \sO_A(-1)}(1) \oplus \sO) \to \PP^B$ be the natural projection.  The two Chern roots of $\sO_{B\otimes \sO_A(-1)}(1)\oplus \sO$ are $u_B$ and $0$, so
applying Quillen's formula (\ref{Quill-formula}) with $f_1(t) \equiv 1$ being the fundamental class, we have
\begin{equation}\label{term 2}
\pi_{1*}([\PP_{\PP^B}(\sO_{B\otimes \sO(-1)}(1) \oplus \sO)])=\frac{1}{0-_\Omega u_B}+\frac{1}{u_B-_\Omega 0}=\frac{1}{-_\Omega u_B}+\frac{1}{u_B}.
\end{equation}

Next, let $\pi_2: \PP_{\PP(A)} (B\otimes \sO_A(-1) ) \to \PP(A)$ be the projection.
The Chern roots of the bundle $B\otimes \sO_A(-1)$ are $b_i-_\Omega v_A$ $i=1,\dots,k$. Applying Quillen's formula (\ref{Quill-formula}) with $f_2(t) =\frac{1}{-_\Omega t}+\frac{1}{t}$, we get
\begin{align*}
\pi_{2*}(f_2(u_B))
&=\sum_{i=1}^k\frac{(-_\Omega(-_\Omega(b_i-_\Omega v_A)))^{-1}+(-_\Omega(b_i-_\Omega v_A))^{-1}}{\prod_{j\neq i}((b_j-_\Omega v_A)-_\Omega(b_i-_\Omega v_A))}\\
&=\sum_{i=1}^k\frac{(b_i-_\Omega v_A)^{-1}+(-_\Omega b_i+_\Omega v_A)^{-1}}{\prod_{j\neq i}(b_j-_\Omega b_i)}.
\end{align*}

Finally, let $\pi_3: \PP(A) \to Z$ be the projection.   Letting
$f_3(t):=\sum_{i=1}^k\frac{(b_i-_\Omega t)^{-1}+(-_\Omega b_i+_\Omega t)^{-1}}{\prod_{j\neq i}(b_j-_\Omega b_i)}$, then $\pi_*(\PP_{\PP^B}(\sO_{B\otimes \sO_A(-1)}(1) \oplus \sO))=\pi_{3*}(f_3(v_A))$ and
Quillen's formula (\ref{Quill-formula}) yields
\[
\pi_{3*}(f_3(v_A))
=\sum_{i=1}^k\sum_{m=1}^k\frac{(b_i+_\Omega a_m)^{-1}+(-_\Omega b_i-_\Omega a_m)^{-1}}{\prod_{j\neq i}(b_j-_\Omega b_i)\prod_{l\neq m}(a_l-_\Omega a_m)}.
\]

Similarly, for the projective bundle $\pi': \PP_{\PP^A}(\sO_{A\otimes \sO_B(-1)}(1) \oplus \sO) \to Z$, we have:
\[
\pi'_*([\PP_{\PP^A}(\sO_{A\otimes \sO_B(-1)}(1) \oplus \sO)])=\sum_{i=1}^k\sum_{m=1}^k\frac{(a_i+_\Omega b_m)^{-1}+(-_\Omega a_i-_\Omega b_m)^{-1}}{\prod_{j\neq i}(a_j-_\Omega a_i)\prod_{l\neq m}(b_l-_\Omega b_m)}
\]
and the lemma follows.
 \end{proof}

\subsubsection{The term $\PP_{\PP(A)} (B\otimes \sO_A(-1) \oplus \sO)$}

Let $\pi_1: \PP_{\PP(A)} (B\otimes \sO_A(-1) \oplus \sO) \to \PP(A)$ be the projection.
The Chern roots of the bundle $B\otimes \sO_A(-1) \oplus \sO$ are  $b_i-_\Omega v_A$, $i=1,\dots,k$, and $0$.
Applying Quillen's formula (\ref{Quill-formula}) with $f_1(t) \equiv 1$, we get
\[
\pi_{1*}([\PP(B\otimes \sO_A(-1) \oplus \sO)])=\sum_{i=1}^k\frac{1}{(v_A-_\Omega b_i)\prod_{j\neq i}(b_j-_\Omega b_i)}+\frac{1}{\prod_{i=1}^k(b_i-_\Omega v_A)}.
\]

Now let $\pi_2: \PP(A) \to Z$ be the projection.  Letting $f_2(t):=\sum_{i=1}^k[(t-_\Omega b_i)\prod_{j\neq i}(b_j-_\Omega b_i)]^{-1}+(\prod_{i=1}^k(b_i-_\Omega t))^{-1}$, we have $\pi_*([\PP_{\PP(A)} (B\otimes \sO_A(-1) \oplus \sO)])=\pi_{2*}(f_2(v_A))$ and 
Quillen's formula (\ref{Quill-formula}) gives
\[
\pi_{2*}(f_2(v_A))=\sum_{m=1}^k\frac{\sum_{i=1}^k((-_\Omega a_m-_\Omega b_i)\prod_{j\neq i}(b_j-_\Omega b_i))^{-1}+(\prod_{i=1}^k(b_i+_\Omega a_m))^{-1}}{\prod_{l\neq m}(a_l-_\Omega a_m)}\\
\]
Similarly, for the bundle $\pi': \PP_{\PP(B)} (A\otimes \sO_A(-1) \oplus \sO) \to Z$, we have:
\[
\pi'_*([\PP_{\PP(B)}(A\otimes \sO_A(-1) \oplus \sO)])
=\sum_{m=1}^k\frac{\sum_{i=1}^k((-_\Omega b_m-_\Omega a_i)\prod_{j\neq i}(a_j-_\Omega a_i))^{-1}+(\prod_{i=1}^k(a_i+_\Omega b_m))^{-1}}{\prod_{l\neq m}(b_l-_\Omega b_m)}
\]
Therefore,
\begin{multline*}
\pi_*([\PP_{\PP(A)} (B\otimes \sO_A(-1) \oplus \sO)])-\pi'_*([\PP_{\PP(B)} (A\otimes \sO_B(-1) \oplus \sO)])\\
=\sum_{m=1}^k\left(\frac{1}{\prod_{i=1}^k(b_i+_\Omega a_m)\prod_{l\neq m}(a_l-_\Omega a_m)}-\frac{1}{\prod_{i=1}^k(a_i+_\Omega b_m)\prod_{l\neq m}(b_l-_\Omega b_m)}\right).
\end{multline*}

This finishes the proof of proposition~\ref{prop:Diff_Omega}.

\subsection{Flops in the elliptic cohomology ring}\label{sec: flop in ell}
In this subsection, we prove the following Proposition.
\begin{prop}\label{prop: flop in ell}  Suppose $X_1$ and $X_2$ are smooth projective varieties related by a flop.
Notations as above, in the ring $\Ell^*_\QQ(Z)$, we have 
\[
[\PP_{\PP(A)} (B\otimes \sO_A(-1) \oplus \sO)]=[\PP_{\PP(B)} (A\otimes \sO_B(-1) \oplus \sO)].
\] In particular, we have $[X_1]-[X_2]=0$ in $\Ell_\QQ^*(Z)$, and hence $[X_1]-[X_2]=0$ in $\tilde{\Ell}[1/p]^*(k)$, where $p$ is the exponential characteristic of $k$.
\end{prop}

\begin{remark} Once we know that $[X_1]-[X_2]=0$ in $\Ell_\QQ^*(Z)$, it follows by pushing forward to $\Spec k$ that
$[X_1]-[X_2]=0$ in $\Ell^*_\QQ(k)$. But $[X_1]-[X_2]$ is a well-defined element in $\tilde{\Ell}[1/p]^*(k)$ and $\tilde{\Ell}[1/p]^*(k)\to \Ell_\QQ^*(k)$ is injective, since $\tilde{\Ell}[1/p]^*(k)$ is by construction an integral domain, hence $[X_1]-[X_2]=0$ in $\tilde{\Ell}[1/p]^*(k)$, as claimed above.
\end{remark}

\begin{proof}[Proof of the proposition] We retain the notation of the previous section.
Thanks to Proposition~\ref{prop:Diff_Omega}, we can reduce $[X_1]-[X_2]$ to an explicit element in $\MGL^*(Z)$.
Applying the canonical map $\MGL^*(Z)\to \Ell[1/p]^*(Z)$ to this element,
we have, in the ring $\Ell[1/p]^*(Z)$,
\[
[X_1]-[X_2]
=\sum_{m=1}^k\left(\frac{1}{\prod_{i=1}^k(b_i+_\Omega a_m)\prod_{l\neq m}(a_l-_\Omega a_m)}-\frac{1}{\prod_{i=1}^k(a_i+_\Omega b_m)\prod_{l\neq m}(b_l-_\Omega b_m)}\right).
\]  
We would like to show the above expression is   $0$ in $\Ell^*_\QQ(Z)$.

Recall \eqref{eqn:EllipticExp} the exponential $\lambda(t)$ of our formal group law. Let $A_i:=\lambda^{-1}(a_i)$, and $B_i:=\lambda^{-1}(b_i)$, $i=1,\dots,k$. Then in $\Ell^*_\QQ(Z)$, $[X_1]-[X_2]$ becomes:
\[
\sum_{m=1}^k\left(\prod_{i=1}^k\frac{1}{\lambda}(B_i+ A_m)\prod_{l\neq m}\frac{1}{\lambda}(A_l- A_m)-\prod_{i=1}^k\frac{1}{\lambda}(A_i+ B_m)\prod_{l\neq m}\frac{1}{\lambda}(B_l- B_m)\right).
\]

We expand $1/\lambda$ using the definition
\[\frac{1}{\lambda}(t)=\frac{Q(t)}{t}=\frac{1}{2\pi i}e^{kt}e^{\zeta(z)\frac{t}{2\pi i}}\frac{\sigma(\frac{t}{2\pi i}-z, \tau)}{\sigma(\frac{t}{2\pi i}, \tau)\sigma(-z, \tau)}. 
\]
After factoring out some obvious common terms and deleting these, 
$[X_1]-[X_2]$ becomes
\[
\sum_{m=1}^k\left(\prod_{i=1}^k\frac{\sigma(B_i+  A_m-z)}{\sigma(B_i+ A_m)}\prod_{l\neq m}\frac{\sigma(A_l-A_m-z)}{\sigma(A_l-A_m)}-\prod_{i=1}^k\frac{\sigma(A_i+ B_m-z)}{\sigma(A_i+ B_m)}\prod_{l\neq m}\frac{\sigma(B_l-B_m-z)}{\sigma(B_l- B_m)}\right).
\]

Now let
\[
x_i=A_i, x_{k+j}=-B_j, y_i=A_i-z,y_{k+j}=-B_j+z;\  i,j=1,\dots,k.
\]
Using the fact that $\sigma(z)$ is an odd function, we get, for all $m=1,\dots,k$,
\[
\sigma(-z)\prod_{i=1}^k\frac{\sigma(B_i+ A_m-z)}{\sigma(B_i+ A_m)}\prod_{l\neq m}\frac{\sigma(A_l-  A_m-z)}{\sigma(A_l- A_m)}
=\frac{\prod_{j=1}^k\sigma(y_{k+j}-x_m)\prod_{l=1}^k\sigma(y_l-x_m)}{\prod_{j=1}^k\sigma(x_{k+j}-x_m)\prod_{l\neq m}\sigma(x_l-x_m)}
\]
and
\[
-\sigma(-z)\prod_{i=1}^k\frac{\sigma(A_i+ B_m-z)}{\sigma(A_i+ B_m)}\prod_{l\neq m}\frac{\sigma(B_l-  B_m-z)}{\sigma(B_l-B_m)}
=\frac{\prod_{i=1}^k\sigma(y_i-x_{k+m})\prod_{l=1}^k\sigma(y_{k+l}-x_{k+m})}{\prod_{i=1}^k\sigma(x_i-x_{k+m})\prod_{l\neq m}\sigma(x_{k+l}-x_{k+m})}.
\]

The proposition now follows from the following classical identity for the sigma-function (see \S20.53, Example 3 of \cite{WW}):
Assuming $\sum_{r=1}^n x_r=\sum_{r=1}^n y_r$,  we have
\[
\sum_{r=1}^n\frac{\sigma(x_r-y_1)\sigma(x_r-y_2)\cdots\sigma(x_r-y_n)}
{\sigma(x_r-x_1)\sigma(x_r-x_2)\cdots*\cdots\sigma(x_r-x_n)}=0
\]
with the $*$ denoting that the  term $\sigma(x_r-x_r)$ is to be omitted.
\end{proof}

\section{The algebraic elliptic cohomology ring with rational coefficients}\label{sec:FlopIdeal}

Let $\sI_\fl$ be the ideal in $\MGL_\QQ^*(k)$ generated by differences $[X_1]-[X_2]$, where $X_1$ and $X_2$ are related by a flop and $\sI_\clf\subseteq \sI_\fl$  the ideal in $\MGL_\QQ^*(k)$ generated by differences $[X_1]-[X_2]$, where $X_1$ and $X_2$ are related by a  classical flop. Section~\ref{sec: flop in ell} shows that the elliptic genus $\phi:\MGL^*(k)\to\Ell[1/p]^*(k)$ factors through the quotient $\MGL^*(k)/\sI_\fl$.

The next proposition is proved in \cite[\S5]{T}.  However, as Totaro's proof uses some non-algebraic objects, we give a completely algebraic (albeit quite similar) proof in an appendix.

\begin{prop}\label{prop:Ideal}
The ideal $\sI_\clf$ in $\MGL^*_\QQ(k)$ contains a system of polynomial generators $x_n$ of $\MGL^*_\QQ(k)$ in degree $n\le -5$.
\end{prop}

Our main results in this section are:

\begin{prop}\label{prop: deg<5}
The classifying map  $\phi_E:\MGL^*(k)\to  \QQ(\!(e^{2\pi iz})\!)[\![e^{2i\pi\tau},k]\!]$ descends to define an isomorphism of
$\MGL_\QQ^*(k)/\sI_\clf$ with the polynomial subalgebra   $R_\QQ=\QQ[a_1, a_2, a_3, a_4]$ of $\QQ(\!(e^{2\pi iz})\!)[\![e^{2i\pi\tau},k]\!]$, with $a_i$ in degree $-i$.
\end{prop}

This implies the following corollary.
\begin{corollary}\label{cor: injModFlop}
The natural ring homomorphism $\MGL_\QQ^*(k)/\sI_\clf\to \Ell_\QQ^*(k)$ is injective and $\sI_\clf=\sI_\fl$.
\end{corollary}

\begin{proof}[Proof of the corollary, assuming the proposition] The first assertion follows immediately from proposition~\ref{prop: deg<5}; the second from the first, noting that $\MGL_\QQ^*(k)/\sI_\clf\to \Ell_\QQ^*(k)$ factors through the quotient $\MGL_\QQ^*(k)/\sI_\fl$ by proposition~\ref{prop: flop in ell}.
\end{proof}

\begin{proof}[Proof of proposition~\ref{prop: deg<5}]
Following H\"ohn \cite{H}, we define the four elements $W_i$, $i=1, 2, 3, 4$ in $\MGL^*(k)$ via  their Chern numbers as follows,
\begin{align*}
&c_1[W_1]=2;\\
&c_1^2[W_2]=0, c_2[W_2]=24;\\
&c_1^3[W_3]=0, c_1c_2[W_3]=0, c_3[W_3]=2;
\\
&c_1^4[W_4]=0, c_1^2c_2[W_4]=0, c_2^2[W_4]=2, c_1c_3[W_4]=0, c_4[W_4]=6.
\end{align*}
In fact, solving a system of linear equations, one can write down $W_i$ explicitly as rational linear combinations of products of projective spaces (in $\MGL^*(k)$).
\begin{align*}
W_1&:=[\PP^1];\\
W_2&:=-16[\PP^2]+18[\PP^1\times \PP^1];\\
W_3&:=\frac{3}{2}[\PP^3]-4[\PP^2\times \PP^1]+\frac{5}{2}[(\PP^1)^3];\\
W_4&:=-4[\PP^4]+12.5[\PP^3\times \PP^1]+6[\PP^2\times \PP^2]-26[\PP^2\times (\PP^1)^2]+11.5(\PP^1)^4.
\end{align*}

Write the Hirzebruch characteristic power series as $Q(t)=1+f_1t+f_2t^2+\cdots$.  It follows easily from the fact that $\sigma(z)$ is an odd function with $\sigma'(0)=1$ that $f_1=k$. Let $A$, $B,$ $C$, and $D$ be such that
\begin{align*}
f_1&=\frac{1}{2}A;\\
f_2&={\frac{1}{2^4\cdot3}}(6A^2-B);\\
f_3&={\frac{1}{2^5\cdot3}}(2A^3- AB + 16C);\\
f_4&={\frac{1}{2^9 \cdot 3^2 \cdot 5}}(60A^4-60A^2B + 1920AC + 7B^2 -1152D).
\end{align*}
Let $K_i$ denote the homogeneous degree $i$ part of the elliptic genus $\phi$.   A calculation shows (see, e.g., \S2.2 of \cite{H}) that
\begin{align*}
K_1&=\frac{1}{2}Ac_1;\\
K_2&={\frac{1}{2^4\cdot3}}((6A^2-B)c_1^2+2Bc_2);\\
K_3&={\frac{1}{2^5\cdot3}}((2A^3- AB + 16C)c_1^3+(2AB-48C)c_2c_1+48Cc_3);\\
K_4&={\frac{1}{2^9 \cdot 3^2 \cdot 5}}((60A^4-60A^2B + 1920AC + 7B^2 -1152D)c_1^4\\
&\hskip15pt+(24B^2-2304D)c_2^2+(120A^2B-5760AC-28B^2+4608D)c_1^2c_2\\
&\hskip25pt+(5760AC+8B^2-4608D)c_3c_1+(-8B^2+4608D)c_4).
\end{align*}
Therefore, $\phi(W_1)=A=2k$, $\phi(W_2)=B$, $\phi(W_3)=C$ and $\phi(W_4)=D$.

Next, we compare the elements $A, B, C, D$ with the polynomial generators of $\Ell_\QQ^*(k)$ \eqref{eqn:Generators}. 
The same calculations as in  \cite[\S2.5, pg.57-8]{H} (adjusting for our different normalizations)  show that $B=24a_2$, $C=a_3$, and $D=6a_2^2-a_4$, and we have already seen that $A=2k$. Clearly $A, B, C, D\in 
\QQ[a_1, a_2, a_3, a_4]$ generate the entire polynomial ring, so $\phi:\MGL_\QQ^*(k)/\sI_\clf\to R_\QQ$ is surjective. By proposition~\ref{prop:Ideal}, $\MGL_\QQ^*(k)/\sI_\clf$ is a quotient of a weighted polynomial ring $\QQ[x_1, x_2, x_3, x_4]$ with $\deg x_i=-i$, so by reason of dimension, both the quotient map  $\QQ[x_1, x_2, x_3, x_4]\to \MGL_\QQ^*(k)/\sI_\clf$ and  $\phi:\MGL_\QQ^*(k)/\sI_\clf\to R_\QQ$ are isomorphisms, completing the proof.
\end{proof}

\section{Birational symplectic varieties}\label{sec:Birat}
In this section we work over a base field $k$ of characteristic zero.
\subsection{Specialization in algebraic cobordism theory}
We will use the specialization morphism in algebraic cobordism theory. The existence of a specialization morphism in algebraic cobordism theory is folklore; for lack of a reference, we sketch a construction here.
\begin{prop}
Let $C$ be a smooth curve, and $p:\calX\to C$ a smooth projective morphism. Let $o\in C$ be a closed point with fiber $\calX_o$ and $\eta$ be the generic point of $C$ whose fiber is denoted by $\calX_\eta$. Let $i:\calX_o\to\calX$ and $j:\calX_\eta\to\calX$ be the natural embeddings. Then there is a natural morphism $\sigma:\Omega^*(\calX_\eta)\to\Omega^*(\calX_o)$ such that $\sigma\circ j^*=i^*$ where $j^*$ is the pull-back and $i^*$ is the Gysin morphism.
\end{prop}
\begin{proof}
Let $R$ be the local ring of $C$ at $o\in C$.  Although in this case neither $\sX_R$ nor $\sX_\eta$ are $k$-schemes of finite type, they are both projective limits of such, which allows us to define $\Omega^*(\sX_R)$ and $\Omega^*(\sX_\eta)$ as
\[
\Omega^*(\sX_R):=\lim_{0\in U\subset C}\Omega^*(p^{-1}(U));\quad \Omega^*(\sX_\eta):=\lim_{\0\neq U\subset C}\Omega^*(p^{-1}(U)).
\]
Here $U\subset C$ is an open subscheme. We may then replace $C$ with $\Spec R$, $\sX$ with $\sX_R$.

For any integer $n$, there is a localization short exact sequence
\[
\Omega^{n-1}(\calX_o)\xrightarrow{i_*}\Omega^{n}(\calX)\xrightarrow{j^*}\Omega^{n}(\calX_\eta)\to 0.
\]

Let $i^*:\Omega^{n}(\calX)\to\Omega^n(\calX_o)$ be the pull-back. In order to show it factors through $j^*$, it suffices to check that $i^*\circ i_*=0$. This is true since $i^*\circ i_*\cong c_1\sO_{\calX_o}|_{\calX_o}=0$ (see \cite[Lemma 3.1.8]{LM}).
\end{proof}

We note that the specialization map $\sigma:\Omega^*(\calX_\eta)\to\Omega^*(\calX_o)$ is a ring homomorphism, and is natural with respect to pullback and push-forward in the following sense: Let $q:\sY\to C$ be a smooth projective morphism, with $C$ as above,  let $f:\sY\to \sX$ be a morphism over $C$ and let $f_o:\sY_o\to \sX_o$, $f_\eta:\sY_\eta\to \sX_\eta$  be the respective  restrictions of $f$.  Let $\sigma_\sX:\Omega^*(\sX_\eta)\to\Omega^*(\sX_o)$, $\sigma_\sY:\Omega^*(\sY_\eta)\to\Omega^*(\sY_o)$ be the respective specialization maps. Then
\begin{enumerate}
\item $f_o^*\circ \sigma_\sX=\sigma_\sY\circ f_\eta^*$.
\item Suppose $f$ is projective. Then $f_{o*}\circ \sigma_\sY=\sigma_\sX\circ f_{\eta*}$.
\end{enumerate}
Indeed, as the respective restriction maps $j^*$ are surjective, the fact that $\sigma$ is a ring homomorphism and the compatibility (1) follows from the fact that pullback maps are functorial ring homomorphisms. For (2), we note that the diagram
\[
\xymatrix{
\sY_o\ar[r]^{i_\sY}\ar[d]_{f_o}&\sY\ar[d]^f\\
\sX_o\ar[r]_{i_\sX}&\sX
}
\]
is cartesian, and then the compatibility (2) follows from the base-change identity $i_{\sX}^*\circ f_*=f_{o*}\circ i_\sY^*$ and the surjectivity mentioned above.

\subsection{Cobordism ring of birational symplectic varieties}
Consider two birational irreducible symplectic varieties $X_1$ and $X_2$ satisfying the following condition:
There exist smooth projective algebraic varieties $\mathcal{X}_1$ and  $\mathcal{X}_2$, flat over a smooth quasi-projective curve $C$ with a closed point $o\in C$, such that:

\begin{romenum}
  \item \label{cond:symplectic1} the fiber of $\mathcal{X}_i$ over $o$ is $(\mathcal{X}_i)_o=X_i$;
  \item \label{cond:symplectic2} there is an isomorphism $\Psi: (\mathcal{X}_1)_{C\setminus \{o\}} \to (\mathcal{X}_2)_{C\setminus \{o\}}$ over $C$.
\end{romenum}

The counterpart of the following Proposition in Chow theory is proved in \cite{G13}.
Let $\Omega^*$ be the algebraic cobordism with $\Z$ coefficient. When the base field $k$ has characteristic zero, then $\Omega^*=\MGL^*$.
\begin{prop}\label{prop:SympBirat}
Let $X_1$ and $X_2$ be two birational symplectic varieties satisfying conditions \eqref{cond:symplectic1} and \eqref{cond:symplectic2}. The deformation data induce an isomorphism
\[
\Omega^*(X_1) \cong \Omega^*(X_2).
\]
\end{prop}

It was proved in \cite{Huy97} that the above conditions \eqref{cond:symplectic1} and \eqref{cond:symplectic2} hold, when:
\begin{itemize}
  \item either $X_1$ and $X_2$ are connected by a general Mukai flop,
  \item or $X_1$ and $X_2$ are isomorphic in codimension two, that is, there exist isomorphic open subsets $U_1\subset X_1$, and $U_2\subset X_2$ with
  $\codim_{X_i}(X_i\setminus U_i)\geq 3$, for $i=1, 2$.
\end{itemize}

The proof follows the same idea as in \cite{G13}, nevertheless, for the convenience of the readers, we include the proof.

Let $\sigma_i:\Omega^*(\mathcal{X}_{i \eta})\to\Omega^*(\mathcal{X}_{i o})$ be the specialization map, where $i=1, 2$. Let $\Delta\subset (\mathcal{X}_{1 \eta}\times \mathcal{X}_{2 \eta})$ be the diagonal, or the graph of the isomorphism $\Psi$ as in condition~\eqref{cond:symplectic2} restricted to the generic fiber.
We define an element \[Z:=(\sigma_1\times \sigma_2)([\Delta])\in \Omega^*(X_1\times X_2).\]
The projection $X_1\times X_2\to X_i$ for $i=1$, $2$ is denoted by $p_i$.
The element $Z\in\Omega^*(X_1\times X_2)$ defines a map $[Z]: \Omega^*(X_1)\to \Omega^*(X_2)$ by $\alpha\mapsto p_{2*}( p_1^*(\alpha)\cap Z)$. By symmetry, we also have a map $[Z^{op}]: \Omega^*(X_2)\to \Omega^*(X_1)$, given by $\beta\mapsto p_{1*}( p_2^*(\beta)\cap Z^{op})$, where $Z^{op}\in \Omega^*(X_2\times X_1)$ is the image of $Z$ by the symmetry morphism $X_1\times X_2\to X_2\times X_1$.

We summarize the notations in the following diagram:
\begin{equation}\label{dia:sympl}
\xymatrix{
& & X_1\times X_2\times X_1 \ar[lld]_{p_{12}}\ar[d]^{p_{13}}\ar[rrd]^{p_{23}}&&\\
X_1\times X_2 \ar[rd]_{p_1}\ar[drr]|-(.35){p_2} |(.63)\hole |(.67)\hole&& X_1\times X_1\ar[ld]_{pr_1}\ar[rd]^{pr_2}&&X_2\times X_1
\ar[ld]^{p_1}\ar[dll]|-(.35){p_2}|(.63)\hole |(.67)\hole\\
& X_1 &X_2& X_1 &\\
}\end{equation}

Now we check that $[Z^{op}]\circ [Z]=1$, i.e., for any $\alpha\in\Omega^*(X_1)$ we have \[p_{1*}\bigg( p_2^*\big(p_{2*}( p_1^*\alpha\cap Z)\big)\cap Z^{op}\bigg)=\alpha.\] The diagonal in $X_i\times X_i$ will be denoted by $\Delta_{X_i}$. Similarly we have $\Delta_{\mathcal{X}_{i\eta}}\subseteq \mathcal{X}_{i\eta}\times\mathcal{X}_{i\eta}$ as the diagonal.
We have:
\begin{align*}
p_{1*}\bigg( p_2^*\big(p_{2*}( p_1^*\alpha\cap Z)\big)\cap Z^{op}\bigg)
&=p_{1*}\bigg(\big( p_{23*}\big(p_{12}^*( p_1^*\alpha\cap Z)\big)\big)\cap Z^{op}\bigg)\\
&=(p_{1}p_{23})_*\bigg((p_{13} pr_1)^*\alpha\cap p_{12}^*Z \cap p_{23}^*Z^{op}\bigg)\\
&=(pr_2)_*\bigg(pr_1^*\alpha\cap p_{13*}\big(p_{12}^*Z \cap p_{23}^*Z^{op}\big)\bigg)\\
&=(pr_2)_*(pr_1^*\alpha\cap \Delta_{X_1})\\
&=\alpha.
\end{align*}
These  equalities are all  obvious, with the exception of the fourth one, which follows from the following lemma.
\begin{lemma} Notations as in the diagram, we have:
\[p_{13*}(p_{12}^*Z \cap p_{23}^*Z^{op})=\Delta_{X_1}\in \Omega(X_1\times X_1).\]
\end{lemma}
\begin{proof}
Consider the same diagram as \eqref{dia:sympl} with $X_i$ replaced by $\mathcal{X}_{i\eta}$. We make the convention here in the proof that for any map $p$ in diagram~\eqref{dia:sympl}, the corresponding map for generic fibers will be denoted by $\widetilde{p}$.
Note that
\[\widetilde{p_{13}}_*(\widetilde{p_{12}}^*\Delta \cap \widetilde{p_{23}}^*\Delta^{op})=\Delta_{\mathcal{X}_{1\eta}} \in \Omega((\mathcal{X}_1)_{\eta}\times (\mathcal{X}_1))_{\eta}.\]
Applying the specialization map $\sigma$ on both sides, we get:
\[p_{13*}(p_{12}^*Z \cap p_{23}^*Z^{op})=\Delta_{X_1}\in \Omega(X_1\times X_1).\]
\end{proof}

Let $\pi_i: \Omega^*(X_i)\to \Omega^*(k)$ be the structure map of $X_i \to k$, where $i=1, 2$.
\begin{prop}\label{prop:Birat}
Let $[Z]: \Omega^*(X_1)\to \Omega^*(X_2)$ be the map constructed as above. We have,
\begin{enumerate}
  \item $[Z](1_{X_1})=1_{X_2},$
  \item $\pi_{1*}(1_{X_1})=\pi_{2*}(1_{X_2})$.
\end{enumerate}
\end{prop}

\begin{proof}
For (1) we have
\[
\widetilde{p_{2}}_*(\widetilde{p_1}^*(1_{(\mathcal{X}_1)_{\eta}})\cap \Delta)=
1_{(\mathcal{X}_2)_{\eta}}\in \Omega((\mathcal{X}_2)_{\eta}).
\]
Applying the specialization map $\sigma$ on both sides, we get:
\[
p_{2*}(p_1^*(1_{X_1})\cap Z)=1_{X_2}\in \Omega^*(X_2).
\]

For (2), let $\Psi: \mathcal{X}_{1\eta}\cong \mathcal{X}_{2\eta}$ be the isomorphism as in condition~\eqref{cond:symplectic2},
we have $\Psi_*(1_{\mathcal{X}_{1\eta}})=1_{\mathcal{X}_{2\eta}}$, and $\widetilde{\pi_{1}}_*(1_{\mathcal{X}_{1\eta}}))=\widetilde{\pi_{2}}_*(1_{\mathcal{X}_{2\eta}}))$.
Applying the specialization map, (2) follows.
\end{proof}

For a vector bundle $E$ on some $X\in \Sm_k$, let $c^{i_1,\ldots, i_r}(E)$ denote the product
 $c_{i_1}(E)\cdots c_{i_r}(E)$ in $\CH^*(X)$. For $X$ a smooth projective variety over $k$, the {\em Chern number} $c^I(X)$ associated to an index $I=(i_1,\ldots, i_r)$ with $\sum_ji_j=\dim_kX$ and $i_j>0$ is $\text{deg}_k(c^I(\TT_X))$.

\begin{corollary}\label{cor:ChernNum} Let $X_1$ and $X_2$ be two birational symplectic varieties satisfying conditions \eqref{cond:symplectic1} and \eqref{cond:symplectic2}. Then $X_1$ and $X_2$ have the same Chern numbers.
\end{corollary}

\begin{proof} For an integer $d>0$, let $P_d$ denote the number of partitions of $d$. It is well-known that the function $X\mapsto \prod_I c^I(X)$ on smooth projective irreducible $k$-schemes of dimension $d$ over $k$ descends via the map $X\mapsto [X]\in \Omega^{-d}(k)$ to well-defined   homomorphism $c^{*,d}:\Omega^{-d}(k)\to \Z^{P_d}$. The result is now an immediate consequence of Proposition~\ref{prop:Birat}.
\end{proof}

\appendix
\section{The ideal  generated by differences of flops}
Let $\sI_\clf$ be the ideal in $\MGL^*_\QQ(k)$ generated by those $[X_1]-[X_2]$ with $X_1$ and $X_2$ related by a classical flop.
\begin{prop}
The ideal $\sI_\clf$ in $\MGL^*_\QQ(k)$ contains a system of polynomial generators $x_n$ of $\MGL^*_\QQ(k)$ in all degrees  $n \leq -5$.
\end{prop}

This proposition was originally proved in Section~5 of \cite{T}, using some non-algebraic constructions. Totaro's proof is based on explicit calculations that lend themself to our setting after a slight adjustment.

For a smooth irreducible projective variety $X$ over $k$, we have the characteristic class $s^n(X)=\langle \xi_1^n+\cdots +\xi_n^n,[X]\rangle$, with $\xi_i$ being the Chern roots of the tangent bundle of $X$ (in the Chow ring $\CH^*$). We will use the fact that an element $x$ of $\MGL^{-n}_\QQ(k)$ is a polynomial generator of the ring $\MGL^*_\QQ(k)$ if and only if the Chern number $s^n$ is not zero on $x$ (see e.g., \cite{Adams}).

Following Fulton, we have the $i$th  Segr\'{e} class  of a rank $r$ vector bundle $V$  bundle over a smooth $k$-scheme $X$,   defined as
\[
s_i(V)=\pi_*(u^{i+r-1}),\quad i=0, 1, \ldots,
\] 
where $u=c_1(\sO(1))\in \CH^1(\PP(V))$ and $\pi:\PP(V)\to X$ is the structure morphism. From \cite[Thereom 3.2]{F},  we have
\begin{lemma}\label{lem: segre}
Let $V\to X$ be a vector bundle over a smooth $k$-scheme $X$, let  $s(V)=\sum_is_i(V)$ be the total Segr\'{e} class, $c(V)=\sum_ic_i(V)$ the total Chern class. Then $s(V)=c(V)^{-1}$ in $\CH^*(X)$.
\end{lemma}

Recall that we have shown that, in $\MGL^*(k)$,
\[
[X_1]-[X_2]=[\PP_{\PP(A)} (B\otimes \sO_{\PP(A)}(-1) \oplus \sO)]-[\PP_{\PP(B)} (A\otimes \sO_{\PP(B)}(-1) \oplus \sO)].
\]
For each smooth $Z$ and rank-2 vector bundles $A$ and $B$, there is a pair $X_1$ and $X_2$, related by a classical flop and with exceptional fibers equal to $\PP(A)$ and $\PP(B)$ respectively. Indeed, consider the $\PP^1\times\PP^1$ bundle $q:\PP(A)\times_Z\PP(B)\to Z$, which we embed in the $\PP^3$ bundle $\PP(A\oplus B)\to Z$ via the line bundle  $p_1^*\sO_A(1)\otimes p_2^*\sO_B(1)$.  We then take $Y^0$ to be the affine $Z$ cone in $A\oplus B$ associated to $\PP(A)\times_Z\PP(B)\subset \PP(A\oplus B)$, and $Y$ the closure of $Y^0$ in $\PP(A\oplus B\oplus \sO_Z)$. $Y$ thus contains the $\PP^2$ bundles $P_1:=\PP(A\oplus\sO_Z)$ and $P_2:=\PP(B\oplus\sO_Z)$; we take $X_i\to Y$ to be the blow-up of $Y$ along $P_i$, $i=1, 2$ and $\tilde X$ the blow-up of $Y$ along $Z$.

For each $n\geq5$,  we will find an $(n-3)$-fold $Z$ and rank two vector bundles $A$ and $B$ over $Z$, such that $s^n([\PP_{\PP(A)} (B\otimes \sO_{\PP(A)}(-1) \oplus \sO)])\neq s^n([\PP_{\PP(B)} (A\otimes \sO_{\PP(B)}(-1) \oplus \sO)])$. In fact, we take $Z=\PP^{n-3}$, $A=\sO_Z(1)\oplus \sO_Z$ and $B=\sO_Z^2$.

Set  $h=c_1(\sO_Z(1))$, $v_A:=c_1(\sO_{\PP(A)}(1))$, $v_B=c_1(\sO_{\PP(B)}(1))$, $w_A:=c_1(\sO_{\PP(A\otimes\sO(-1)\oplus \sO)}(1))$,  $w_B=c_1(\sO_{\PP(B\otimes\sO(-1)\oplus 1)}(1))$,  and let  $z_1\cdots,z_{n-3}$ be the Chern roots of  the tangent bundle of $Z=\PP^{n-3}$. Then the Chern roots of the tangent bundle of $\PP_{\PP(A)} (B\otimes \sO_{\PP(A)}(-1) \oplus \sO)$ are $-v_A+w_B, -v_A+w_B, w_B, h+v_A, v_A, z_1, \dots, z_{n-3}$.

For the bundle $\pi_1: \PP_{\PP(A)} (B\otimes \sO_{\PP(A)}(-1) \oplus \sO) \to \PP(A)$, Lemma~\ref{lem: segre} yields
\[
\pi_{1*}(w_B^i)=s_{i-2}(B\otimes \sO_{\PP(A)}(-1) \oplus \sO)=(i-1)v_A^{i-2}.
\]
Similarly, for $\pi_2: \PP(A) \to Z$, we have
\[
\pi_{2*}(v_A^i)=s_{i-1}(A)=\sum_{j=0}^{i-1}(-h)^j.
\]
Using this and the projection formula, we find
\begin{align*}
\pi_{1*}s^n\big([\PP_{\PP(A)} (B\otimes \sO_{\PP(A)}(-1) \oplus \sO)]\big) 
&=\pi_{1*}\left[2(-v_A+w_B)^n\!+\!w_B^n\!+\!(h+v_A)^n\!+\!(v_A)^n\!+\!\sum_{i=1}^{n-3}z_i^n\right]\\
&=2\sum_{i=2}^n{\binom{n}{i}}(-v_A)^{n-i}(i-1)(v_A)^{i-2}+(n-1)(v_A)^{n-2}\\
&=(v_A)^{n-2}(2(-1)^n+n-1).
\end{align*}
According to Lemma~\ref{lem: segre}, $\pi_{2*}(v_A)^{n-2}=(-h)^{n-3}$. Therefore,
\[
\pi_{2*}\pi_{1*}s^n\big(\PP_{\PP(A)} (B\otimes \sO_{\PP(A)}(-1) \oplus \sO) \big)
=h^{n-3}(-2+(-1)^{n-3}(n-1)).
\]

We do the same for the projection $\pi'_1: \PP_{\PP(B)} (A\otimes \sO_{\PP(B)}(-1) \oplus 1) \to \PP(B)$, and $\pi'_2: \PP(B) \to Z$. Using Lemma~\ref{lem: segre}, a calculation  similar to the one above gives
\[
\pi'_{2*}\pi'_{1*}s^n \big(\PP_{\PP(B)} (A\otimes \sO_{\PP(B)}(-1) \oplus 1)\big)
=h^{n-3}(-(n-1)^2+{\binom{n}{2}}+(n-2)(-1)^{n-3})
\]
This gives 
\[
s^n([X_1]-[X_2]) =\frac{n^2-3n-2+2(-1)^{n-1}}{2},
\]
so for $n\geq 5$, $s^n([X_1]-[X_2])\neq 0$, as desired.

\section{$\ell'$-alterations and dualisability (by Jo\"el Riou)}\label{App:B}

\begin{prop}
Let $k$ be a perfect field. Let $\ell$ be a prime number different from the
characteristic of $k$. Then, for any smooth finite type $k$-scheme $U$, the
suspension spectrum $\Sigma^\infty_T U_+$ belongs to the pseudo-abelian
triangulated subcategory of $\SH(k)_{\Z_{(\ell)}}$ generated by the objects
$\Sigma^\infty_T X_+$ where $X$ is projective and smooth over $k$.
\end{prop}

\begin{corollary}\label{cor:Duality}
Let $k$ be a perfect field. Let $p$ denote the caracteristic exponent of
$k$ (i.e., $p>0$ or $p=1$ if the characteristic of $k$ is zero). Then, 
for any smooth finite type $k$-scheme $U$, the
suspension spectrum $\Sigma^\infty_T U_+$ is strongly dualisable in 
$\SH(k)_{\Z[\frac1p]}$.
\end{corollary}

First, we shall see how the corollary follows from the
proposition. The strong dualisability can be formulated using the internal
Hom, which shall be denoted $\SheafHom$ here. An object $A\in
\SH(k)_{\Z[\frac1p]}$ is strongly dualisable if and only if for any object
$M\in \SH(k)_{\Z[\frac1p]}$, the canonical morphism
$\SheafHom(A,\bS_k)\wedge M\to \SheafHom(A,M)$ is an isomorphism. At the
level of stable motivic homotopy sheaves, the localisation functor
$\SH(k)_{\Z[\frac1p]}\to \SH(k)_{\Z_{(\ell)}}$ (for any prime number $\ell$
not dividing $p$) corresponds to the tensor product with $\Z_{(\ell)}$ over
$\Z[\frac1p]$. Then, we see that in order to prove that an object
$A\in\SH(k)_{\Z[\frac1p]}$ is strongly dualisable, it suffices to prove
that its images in all the categories $\SH(k)_{\Z_{(\ell)}}$ (for $\ell$
not dividing $p$) are strongly dualisable. Finally, the corollary follows
from the proposition and the fact that the objects $\Sigma^\infty_T X_+$
are strongly dualisable in $\SH(k)$ if $X$ is projective and
smooth (see \cite[\S2]{Riou} which relies on the work by J.~Ayoub
\cite{Ayoub} and V.~Voevodsky).

In order to prove the proposition, we shall need the following lemma:

\begin{lemma}\label{lemma-degree}
Let $k$ be a field of characteristic $p>0$. Let $\pi\colon V\to U$ be
a finite and
\'etale morphism between connected and smooth $k$-schemes. Let $d$ be the
degree of $\pi$. Then, there exists a dense open subset $U'\subset U$ and a
morphism $s\colon \Sigma^\infty_T U'_+ \to\Sigma^\infty_T V'_+$ in $\SH(k)$
(where $V':=\pi^{-1}(U')$), such that if we denote $\pi'\colon 
\Sigma^\infty_T V'_+ \to\Sigma^\infty_T U'_+$ the induced morphism, the
composition $\pi'\circ s\in\End_{\SH(k)}(\Sigma^\infty_T U'_+)$ can be
written as $\pi'\circ s=d+\alpha$, where $\alpha$ is a nilpotent
endomorphism of $\Sigma^\infty_T U'_+$ in $\SH(k)$.
\end{lemma}

\begin{proof}
We may observe that the map $s$ of the lemma can be obtained by
application of the functor $a_\sharp\colon \SH(U')\to
\SH(k)$ (see \cite[p.~104]{MV}) where $a\colon U'\to \Spec(k)$ is the obvious
morphism. Indeed, passing to the generic point of $U$, we see that
in order to prove the lemma we may assume that $U=\Spec(k)$ and that $V$
is the spectrum of a finite separable field extension $L$ of $k$. The map
$s$ is constructed in \cite[Lemme~1.9]{Riou} and what we know is that the
composition $\pi'\circ s\in \End_{SH(k)}(\bS_k)$ is of the form $d+\alpha$
where $\alpha$ is an endomorphism that vanishes after base change to a big
enough extension of $k$. We now use the isomorphism
$\End_{SH(k)}(\bS_k)\simeq GW(k)$ from \cite[Corollary 1.24 and Remark
1.26]{Morel}. Using this identification, we see that $\alpha$ belongs to
the kernel of
the rank morphism $GW(k)\to \Z$. Then,
we can conclude using the following lemma:
\end{proof}

\begin{lemma}
Let $k$ be a field of characteristic $p>0$. Let $\alpha\in GW(k)$ be an
element in the kernel of the rank morphism $GW(k)\to \Z$. Then, $\alpha$ is
nilpotent in $GW(k)$.
\end{lemma}

\begin{proof}
As the set of nilpotent elements in the commutative ring $GW(k)$ is an
ideal, we may assume $\alpha=\<t\>-1$ where $t\in
k^\times$. We have $(1+\alpha)^2=\<t^2\>=1$, so that $\alpha^2=-2\alpha$. By
induction, we get $\alpha^n=(-2)^{n-1}\alpha$ for $n\geq 1$: we have
to show that $\alpha$ is annihilated by a power of two. If $p=2$,
$2\alpha=0$ holds (see \cite[Lemma~3.9]{Morel}), i.e. $\alpha^2=0$.
Now we assume $p\geq 3$ so that
there is no danger thinking in terms of usual quadratic forms. We first
consider
$\gamma:=\<-1\>-1\in GW(\F_p)$. The quadratic form $-x^2-y^2$ over $\F_p$
represents $1$ (see \cite[Proposition~4, \S{}IV.1.7]{Serre}) so that
$\<-1\>+\<-1\>=\<1\>+\<1\>\in GW(\F_p)$, i.e. $2\gamma=0\in GW(\F_p)$,
which gives $\gamma^2=0$. Let $t\in k^\times$ be any nonzero element in an
extension $k$ of $\F_p$. The quadratic form $q(x,y):=x^2-y^2=(x+y)(x-y)$
represents $t$ (this is $q(\frac{1+t}{2},\frac{1-t}{2})$), which easily
implies that $\<1\>+\<-1\>=\<t\>+\<-t\>$. This is equivalent to saying
$(2+\gamma)\alpha=0\in GW(k)$. It follows that
$4\alpha=(2-\gamma)(2+\gamma)\alpha=0$, and then $\alpha^3=0$.
\end{proof}

\medskip

We can now prove the proposition. It was already proven in the case $k$ is
of characteristic $0$ in \cite{Riou} using Hironaka's resolution of
singularities. In characteristic $p>0$, an argument using de Jong's
alterations also led to the same result with rational coefficients. The
idea was to do an induction on the dimension of the variety $U$, which can
be assumed connected. Then, the property we want to prove becomes a
birational property of $U$, so that we can shrink $U$ if needed. Using the
lemma~\ref{lemma-degree}, we obtain that $\Sigma^\infty_T U'_+$ is a direct
factor of $\Sigma^\infty_T V'_+$ if the degree $d$ has been inverted in the
coefficient ring. If $d$ is invertible in the coefficient ring and if $V'$
is an open subset of a projective and smooth variety, we can deduce the
expected property for $U$. The theorems by de Jong on alterations were
sufficient to conclude in the case of rational coefficients. Here, for
$\Z_{(\ell)}$-coefficients, we have
to ensure that an appropriate alteration can be found with a
prime-to-$\ell$ degree $d$: this is possible thanks to Gabber's improvement
\cite[X~3.5]{Gabber} of de Jong's results.

\newcommand{\arxiv}[1]
{\texttt{\href{http://arxiv.org/abs/#1}{arXiv:#1}}}
\newcommand{\doi}[1]
{\texttt{\href{http://dx.doi.org/#1}{doi:#1}}}
\renewcommand{\MR}[1]
{\href{http://www.ams.org/mathscinet-getitem?mr=#1}{MR#1}}

\end{document}